\documentclass[11pt,onecolumn]{article}

\usepackage{amsmath} 
\usepackage{amssymb}  
\usepackage{amsthm}  
\usepackage{graphicx,psfrag}
\usepackage{cite}

\newtheorem{definition}{Definition}
\newtheorem{remark}{Remark}
\newtheorem{example}{Example}

\newtheorem{theorem}{Theorem}

\newtheorem{problem}{Problem}
\newtheorem{proposition}{Proposition}

\usepackage{color}

\title{\LARGE \bf
On Lossless Approximations, the Fluctuation-Dissipation Theorem, \\ and Limitations of Measurements}

\author{Henrik Sandberg, Jean-Charles Delvenne, and John C. Doyle 
\thanks{H. Sandberg is with
the School of Electrical Engineering, Royal Institute of Technology (KTH), Stockholm, Sweden, {\tt\small hsan@ee.kth.se}.
Supported in part by the Swedish Research Council and the Swedish Foundation for Strategic Research.}
\thanks{J.-C. Delvenne is with the University of Namur (FUNDP),
Department of Mathematics, Namur, Belgium, {\tt\small  jean-charles.delvenne@math.fundp.ac.be}.
Supported in part by the Belgian Programme on Interuniversity Attraction Poles DYSCO, initiated
by the Belgian Federal Science Policy Office. The scientific responsibility rests with its authors.}
\thanks{J.C. Doyle is with California Institute of Technology, Control and Dynamical Systems, M/C 107-81,
Pasadena, CA 91125, USA, {\tt doyle@cds.caltech.edu}.
Supported  by grants NSF-EFRI-0735956, AFOSR-FA9550-08-1-0043, and ONR-MURI-N00014-08-1-0747.}
}

\begin{document}
\maketitle
\thispagestyle{empty}

\begin{abstract}
In this paper, we take a control-theoretic approach to answering
some standard questions in statistical mechanics, and use the
results to derive limitations of classical measurements. A central
problem is the relation between systems which appear
macroscopically dissipative but are microscopically lossless. We
show that a linear system is dissipative if, and only if, it can
be approximated by a linear lossless system over arbitrarily long
time intervals. Hence lossless systems are in this sense dense in
dissipative systems. A linear active system can be approximated by
a nonlinear lossless system that is charged with initial energy.
As a by-product, we obtain mechanisms explaining the Onsager
relations from time-reversible lossless approximations, and the
fluctuation-dissipation theorem from uncertainty in the initial
state of the lossless system. The results are applied to
measurement devices and are used to quantify limits on the
so-called observer effect, also called \emph{back action}, which
is the impact the measurement device has on the observed system.
In particular, it is shown that deterministic back action can be
compensated by using active elements, whereas stochastic back
action is unavoidable and depends on the temperature of the
measurement device.
\end{abstract}

\section{Introduction}
Analysis and derivation of limitations on what is achievable are
at the core of many branches of engineering, and thus of
tremendous importance. Examples can be found in estimation,
information, and control theories. In estimation theory, the
Cram\'er-Rao inequality gives a lower bound on the covariance of
the estimation error, in information theory Shannon showed that
the channel capacity gives an upper limit on the communication
rate, and in control theory Bode's sensitivity integral bounds
achievable control performance. For an overview of limitations in
control and estimation, see the book \cite{Seron+97}. Technology
from all of these branches of engineering is used in parallel in
modern networked control systems \cite{murray+03}. Much research
effort is currently spent on understanding how the limitations
from these fields interact. In particular, much effort has been
spent on merging limitations from control and information theory,
see for example \cite{nair+04,Tatikonda+04,martins+07}. This has
yielded insight about how future control systems should be
designed to maximize their performance and robustness.

Derivation of limitations is also at the core of physics.
Well-known examples are the laws of thermodynamics in classical
physics and the uncertainty principle in quantum mechanics
\cite{wannier,kittel+80,ma85}. The exact implications of these
physical limitations on the performance of control systems have
received little attention, even though all components of a control
system, such as actuators, sensors, and computers, are built from
physical components which are constrained by physical laws.
Control engineers discuss limitations in terms of location of
unstable plant poles and zeros, saturation limits of actuators,
and more recently channel capacity in feedback loops. But how does
the amount of available energy limit the possible bandwidth of a
control system? How does the ambient temperature affect the
estimation error of an observer? How well can you implement a
desired ideal behavior using physical components? The main goal of
this paper is to develop a theoretical framework where questions
such as these can be answered, and initially to derive limitations
on measurements using basic laws from classical physics. Quantum
mechanics is not used in this paper.

The derivation of physical limitations broaden our understanding
of control engineering, but these limitations are also potentially
useful outside of the traditional control-engineering community.
In the physics community, the rigorous error analysis we provide
could help in the analysis of far-from-equilibrium systems when
time, energy, and degrees of freedom are limited. For
Micro-Electro-Mechanical Systems (MEMS), the limitation we derive
on measurements can be of significant importance since the
physical scale of micro machines is so small. In systems biology,
limits on control performance due to molecular implementation have
been studied \cite{VinnicombePaulsson}. It is hoped that this
paper will be a first step in a unified theoretical foundation for
such problems.

\subsection{Related work}
The derivation of thermodynamics as a theory of large systems
which are microscopically governed by lossless and time-reversible
fundamental laws of physics (classical or quantum mechanics) has a
large literature and tremendous progress for over a century within
the field of statistical physics. See for instance
\cite{zwanzig73,ford+kac86,caldeira+83,Lebowitz99} for physicists'
account of how dissipation can appear from time-reversible
dynamics, and the books \cite{wannier,kittel+80,ma85} on
traditional statistical physics. In non-equilibrium statistical
mechanics, the focus has traditionally been on dynamical systems
close to equilibrium. A result of major importance is the
\emph{fluctuation-dissipation theorem}, which plays an important
role in this paper. The origin of this theorem goes back to
Nyquist's and Johnson's work \cite{johnson28,nyquist28} on thermal
noise in electrical circuits. In its full generality, the theorem
was first stated in \cite{cellen+51}; see also \cite{Kubo}. The
theorem shows that thermal fluctuations of systems close to
equilibrium determines how the system dissipates energy when
perturbed. The result can be used in two different ways: By
observing the fluctuation of a system you can determine its
dynamic response to perturbations; or by making small
perturbations to the system you can determine its noise
properties. The result has found wide-spread use in many areas
such as fluid mechanics, but also in the circuit community, see
for example \cite{twiss55,anderson82}.  A recent survey article
about the fluctuation-dissipation theorem is \cite{marconi+08}.
Obtaining general results for dynamical systems far away from
equilibrium (far-from-equilibrium statistical mechanics) has
proved much more difficult. In recent years, the so-called
\emph{fluctuation theorem} \cite{evans+02,lebowitz+99}, has
received a great deal of interest. The fluctuation theorem
quantifies the probability that a system far away from equilibrium
violates the second law of thermodynamics. Not surprisingly, for
longer time intervals, this probability is exceedingly small. A
surprising fact is that the fluctuation theorem implies the
fluctuation-dissipation theorem when applied to systems close to
equilibrium \cite{lebowitz+99}. The fluctuation theorem is not
treated in this paper, but is an interesting topic for future
work.

From a control theorist's perspective, it remains to understand
what these results imply in a control-theoretical setting. One
contribution of this paper is to highlight the importance of the
fluctuation-dissipation theorem in control engineering.
Furthermore, additional theory is needed that is both
mathematically more rigorous and applies to systems not merely
far-from-equilibrium, but maintained there using active control.
More quantitative convergence and error analysis is also needed
for systems not asymptotically large, such as arise in biology,
microelectronics, and micromechanical systems.

Substantial work has already been done in the control community in
formulating various results of classical thermodynamics in a more
mathematical framework. In \cite{Brockett+78,Brockett99}, the
second law of thermodynamics is derived and a control-theoretic
heat engine is obtained (in \cite{san07a} these results are
generalized). In \cite{Haddad+05}, a rigorous dynamical systems
approach is taken to derive the laws of thermodynamics using the
framework of dissipative systems \cite{willems72A,willems72B}. In
\cite{Mitter+05}, it is shown how the entropy flows in Kalman-Bucy
filters, and in \cite{san07d} Linear-Quadratic-Gaussian control
theory is used to construct heat engines. In
\cite{Bernstein+02,barahona+02,san07c}, the problem of how
lossless systems can appear dissipative (compare with
\cite{zwanzig73,ford+kac86,caldeira+83} above) is discussed using
various perspectives. In \cite{georgiou+08}, how the direction of
time affects the difficulty of controlling a process is discussed.

\subsection{Contribution of the paper}
The first contribution of the paper is that we characterize
systems that can be approximated using linear or nonlinear
lossless systems. We develop a simple, clear control-theoretic
model framework in which the only assumptions on the nature of the
physical systems are conservation of energy and causality, and all
systems are of finite dimension and act on finite time horizons.
We construct high-order lossless systems that approximate
dissipative systems in a systematic manner, and prove that a
linear model is dissipative if, and only if, it is arbitrarily
well approximated by lossless causal linear systems over an
arbitrary long time horizon. We show how the error between the
systems depend on the number of states in the approximation and
the length of the time horizon (Theorems~\ref{prop:errbound} and
\ref{thm:diss2}). Since human experience and technology is limited
in time, space, and resolution, there are limits to directly
distinguishing between a low-order macroscopic dissipative system
and a high-order lossless approximation. This result is important
since it shows exactly what macroscopic behaviors we can implement
with linear lossless systems, and how many states are needed. In
order to approximate an {\em active} system, even a linear one,
with a lossless system, we show that the approximation {\em must
be nonlinear}. Note that active components are at the heart of
biology and all modern technology, in amplification, digital
electronics, signal transduction, etc. In the paper, we construct
one class of low-order lossless nonlinear approximations and show
how the approximation error depends on the initial available
energy (Theorems~\ref{thm:nonlin_bound} and
\ref{thm:nonlin_bound2}). Thus in this control-theoretic context,
nonlinearity is not a source of complexity, but rather an
essential and valuable resource for engineering design. These
result are all of theoretical interest, but should also be of
practical interest. In particular, the results give constructive
methods for implementing desired dynamical systems using finite
number of lossless components when resources such as time and
energy are limited.

As a by-product of this contribution, the fluctuation-dissipation
theorem (Propositions~\ref{prop:fluctdiss} and
\ref{prop:johnsonnoise}) and the Onsager reciprocal relations
(Theorem~\ref{thm:timerev}) easily follows. The lossless systems
studied here are consistent with classical physics since they
conserve energy. If time reversibility (see \cite{willems72B} and
also Definition~\ref{def:reversible}) of the linear lossless
approximation is assumed, the Onsager relations follow.
Uncertainty in the initial state of linear lossless approximations
give a simple explanation for noise that can be observed at a
macroscopic level, as quantified by the fluctuation-dissipation
theorem. The fluctuation-dissipation theorem and the Onsager
relations are well know and have been shown in many different
settings. Our contribution here is to give alternative
explanations that use the language and tools familiar to control
theorists.

The second contribution of the paper is that we highlight the
importance of the fluctuation-dissipation theorem for deriving
limitations in control theory. As an application of
control-theoretic relevance, we apply it on models of measurement
devices. With idealized measurement devices that are not lossless,
we show that measurements can be done without perturbing the
measured system. We say these measurement devices have no
\emph{back action}, or alternatively, no \emph{observer effect}.
However, if these ideal measurement devices are implemented using
lossless approximations, simple limitations on the back action
that depends on the surrounding temperature and available energy
emerge. We argue that these lossless measurement devices and the
resulting limitations are better models of what we can actually
implement physically.

We hope this paper is a step towards building a framework for
understanding fundamental limitations in control and estimation
that arise due to the physical implementation of measurement
devices and, eventually, actuation. We defer many important and
difficult issues here such as how to actually model such devices
realistically. It is also clear that this framework would
benefit from a behavioral setting \cite{polderman+97}. However,
for the points we make with this paper, a conventional
input-output setting with only regular interconnections is
sufficient. Aficionados will easily see the generalizations, the
details of which might be an obstacle to readability for others.
Perhaps the most glaring unresolved issue is how to best
motivate the introduction of stochastics. In conventional
statistical mechanics, a stochastic framework is taken for
granted, whereas we ultimately aim to explain if, where, and why
stochastics arise naturally. We hope to address this in future
papers. The paper \cite{san07c} is an early version of this
paper.

\subsection{Organization}
The organization of the paper is as follows: In
Section~\ref{sec:micro}, we derive lossless approximations of
various classes of systems. First we look at memoryless
dissipative systems, then at dissipative systems with memory, and
finally at active systems. In Section~\ref{sec:fluct}, we look at
the influence of the initial state of the lossless approximations,
and derive the fluctuation-dissipation theorem. In
Section~\ref{sec:hardlimit}, we apply the results to measurement
devices, and obtain limits on their performance.

\subsection{Notation}
Most notation used in the paper is standard. Let
$f(t)\in\mathbb{R}^{n \times n}$ and $f_{ij}(t)$ be the $(i,j)$-th
element. Then $f(t)^T$ denotes the transpose of $f(t)$, and
$f(t)^*$ the complex conjugate transpose of $f(t)$. We define
$\|f(t)\|_1 := \sum_{i,j=1}^n |f_{ij}(t)|$, $\|f(t)\|_2 :=
\sqrt{\sum_{i,j=1}^n |f_{ij}(t)|^2}$, and $\bar \sigma (f(t))$ is
the largest singular value of $f(t)$. Furthermore,
$\|f\|_{L_1[0,t]}:= \int_0^t \|f(s)\|_1 ds$, and
$\|f\|_{L_2[0,t]}:= \sqrt{\int_0^t \|f(s)\|_2^2 ds}$. $I_n$ is the
$n$-dimensional identity matrix.

\section{Lossless Approximations}
\label{sec:micro}

\subsection{Lossless systems}
In this paper, linear systems in the form
\begin{equation}
\begin{aligned}
  \dot x(t) & = J x(t) + Bu(t), & x(t) & \in \mathbb{R}^n, \\
  y(t) & = B^T x(t) + Du(t), & u(t),y(t)& \in \mathbb{R}^p,
\end{aligned}
\label{eq:physlinear}
\end{equation}
where $J$ and $D$ are anti symmetric ($J=-J^T$, $D=-D^T$) and
$(J,B)$ is controllable are of special interest. The system
(\ref{eq:physlinear}) is a \emph{linear lossless system}. We
define the \emph{total energy} $E(x)$ of (\ref{eq:physlinear}) as
\begin{equation}
  E(x) := \frac{1}{2} x^Tx.
\label{eq:internalenergy}
\end{equation}
\emph{Lossless} \cite{willems72A,willems72B} means that the total
energy of (\ref{eq:physlinear}) satisfies
\begin{equation}
  \frac{dE(x(t))}{dt} = x(t)^T\dot x(t)= y(t)^Tu(t) =: w(t),
\label{eq:powbalance}
\end{equation}
where  $w(t)$ is the \emph{work rate} on the system. If there is
no work done on the system, $w(t)=0$, then the total energy
$E(x(t))$ is constant. If there is work done on the system,
$w(t)>0$, the total energy increases. The work, however, can be
extracted again, $w(t)<0$, since the energy is conserved and the
system is controllable. In fact, all finite-dimensional linear
minimal lossless systems with supply rate $w(t)=y(t)^Tu(t)$ can be
written in the form (\ref{eq:physlinear}), see
\cite[Theorem~5]{willems72B}. Nonlinear lossless systems will also
be of interest later in the paper. They will also satisfy
(\ref{eq:internalenergy})--(\ref{eq:powbalance}), but their
dynamics are nonlinear. Conservation of energy is a common
assumption on microscopic models in statistical mechanics and in
physics in general \cite{wannier}. The systems
(\ref{eq:physlinear}) are also time reversible if, and only if,
they are also reciprocal, see~\cite[Theorem~8]{willems72B} and
also Definitions~\ref{def:reciprocal}--\ref{def:reversible} in
Section~\ref{sec:memory}. Hence, we argue the systems
(\ref{eq:physlinear}) have desirable ``physical'' properties.
\begin{remark}
In this paper, we only consider systems that are lossless and
dissipative with respect to the supply rate $w(t)=y(t)^Tu(t)$.
This supply rate is of special importance because of its relation
to passivity theory. Indeed, there is a theory for systems with
more general supply rates, see for example
\cite{willems72A,willems72B}, and it is an interesting problem to
generalize the results here to more general supply rates.
\end{remark}
\begin{remark}
The system (\ref{eq:physlinear}) is a linear port-Hamiltonian
system, see for example \cite{Cervera+07}, with no dissipation.
Note that the Hamiltonian of a linear port-Hamiltonian system is
identical to the total energy $E$.
\end{remark}

There are well-known necessary and sufficient conditions for when
a transfer function can be exactly realized using linear lossless
systems: All the poles of the transfer function must be simple,
located on the imaginary axis, and with positive semidefinite
residues, see \cite{willems72B}. In this paper, we show that
linear dissipative systems can be arbitrarily well approximated by
linear lossless systems (\ref{eq:physlinear}) over arbitrarily
large time intervals. Indeed, if we believe that energy is
conserved, then all macroscopic models should be realizable using
lossless systems of possibly large dimension. The linear lossless
systems are rather abstract but have properties that we argue are
reasonable from a physical point of view, as illustrated by the
following example.

\begin{figure}[t]
  \centering
  \psfrag{i}[][][1.0]{$i$}
  \psfrag{v1,c1}[][][1.0]{$v_1,C_1$}
  \psfrag{v2,c2}[][][1.0]{$v_2,C_2$}
  \psfrag{i1,L1}[][][1.0]{$i_1,L_1$}
  \psfrag{+}[][][1.0]{$+$}
  \psfrag{-}[][][1.0]{$-$}
  \includegraphics[width=0.6\hsize]{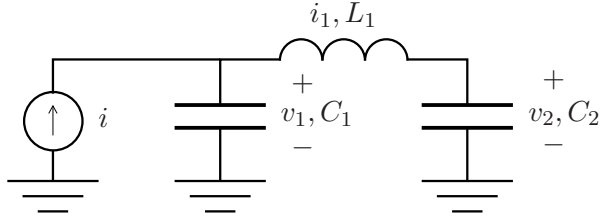}
  \caption{The inductor-capacitor circuit in Example~\ref{ex:1}.}
  \label{fig:LCcircuit}
\end{figure}

\begin{example}
\label{ex:1} It is a simple exercise to show that the circuit in
Fig.~\ref{fig:LCcircuit} with the current $i(t)$ through the
current source as input $u(t)$, and the voltage $v_1(t)$ across
the current source as output $y(t)$ is a lossless linear system.
We have {\small
\begin{align*}
\dot x(t) &=
\begin{pmatrix}
0 & -1/\sqrt{L_1 C_1} & 0 \\
1/\sqrt{L_1 C_1}  & 0 &  -1/\sqrt{L_1 C_2} \\
0 & 1/\sqrt{ L_1 C_2} & 0
\end{pmatrix}
x(t) \\ & \quad +
\begin{pmatrix}
1/\sqrt C_1 \\ 0 \\ 0
\end{pmatrix}
u(t), \\
y(t) & =
\begin{pmatrix} 1/\sqrt C_1 & 0 & 0 \end{pmatrix}
x(t), \\ x(t)^T & = \begin{pmatrix} \sqrt C_1 v_1(t) & \sqrt L_1 i_1(t) & \sqrt C_2 v_2(t) \end{pmatrix}, \\
E(x(t)) &= \frac{1}{2}x(t)^Tx(t) = \frac{1}{2}(C_1v_1(t)^2 + L_1 i_1(t)^2 + C_2 v_2(t)^2), \\ w(t) & = y(t)u(t) = v_1(t)i(t).
\end{align*}}
Note that $E(x(t))$ coincides with the energy stored in the
circuit, and that $w(t)$ is the power into the circuit. Electrical
circuits with only lossless components (capacitors and inductors)
can be realized in the form (\ref{eq:physlinear}), see
\cite{AndersonVong}. Circuits with resistors can always be
\emph{approximated} by systems in the form (\ref{eq:physlinear}),
as is shown in this paper.
\end{example}

\subsection{Lossless approximation of dissipative memoryless systems}
\label{sec:staticio} Many times macroscopic systems, such as
resistors, are modeled by simple static (or memoryless)
input-output relations
\begin{equation}
  y(t) = k u(t),
\label{eq:diss}
\end{equation}
where $k\in\mathbb{R}^{p \times p}$. If $k$ is positive
semidefinite, this system is dissipative since work can never be
extracted and the work rate is always nonnegative, $w(t) =
y(t)^Tu(t)= u(t)^Tku(t)\geq 0$, for all $t$ and $u(t)$. Hence,
(\ref{eq:diss}) is not lossless. Next, we show how we can
approximate (\ref{eq:diss}) arbitrarily well with a lossless
linear system (\ref{eq:physlinear}) over \emph{finite}, but
arbitrarily long, time horizons $[0,\tau]$. First of all, note
that $k$ can be decomposed into $k=k_s + k_a$ where $k_s$ is
symmetric positive semidefinite, and $k_a$ is anti symmetric. We
can use $D=k_a$ in the lossless approximation
(\ref{eq:physlinear}) and need only to consider the symmetric
matrix $k_s$ next.

First, choose the time interval of interest, $[0,\tau]$, and
rewrite $y(t)=k_su(t)$ as the convolution
\begin{equation}
  y(t) = \int_{-\infty}^{\infty} \kappa(t-s) u(s) ds, \quad \kappa(t) := k_s\delta(t),
\label{eq:resmod}
\end{equation}
where $u(t)$ is at least continuous and has support in the
interval $[0,\tau]$,
\begin{equation*}
u(t)= 0, \quad t \in  (-\infty,0] \cup [\tau,\infty),
\end{equation*}
and $\delta(t)$ is the Dirac distribution. The time interval
$[0,\tau]$ should contain all the time instants where we perform
input-output experiments on the system
(\ref{eq:diss})--(\ref{eq:resmod}). The impulse response
$\kappa(t)$ can be formally expanded in a Fourier series over the
interval $[-\tau,\tau]$,
\begin{equation}
  \kappa(t) \sim \frac{k_s}{2\tau}+\sum_{l=1}^{\infty} \frac{k_s}{\tau} \cos l\omega_0 t, \quad \omega_0:=\pi/\tau.
\label{eq:kdistr}
\end{equation}
To be precise, the Fourier series (\ref{eq:kdistr}) converges to
$k_s\delta(t)$ in the sense of distributions. Define the truncated
Fourier series by $\kappa_N(t):= k_s/(2\tau)+\sum_{l=1}^{N-1}
(k_s/\tau) \cos l\omega_0 t$ and split $\kappa_N(t)$ into a causal
and an anti-causal part:
\begin{align*}
  \kappa_N(t) & =: \kappa_N^c(t) + \kappa_N^{ac}(t) \\
  \kappa_N^c(t) & = 0 \quad (t<0),\quad  \kappa_N^{ac}(t) = 0 \quad (t\geq 0).
\end{align*}
The causal part $\kappa_N^c(t)$ can be realized as the impulse
response of a lossless linear system (\ref{eq:physlinear}) of
order $(2N-1)r$ using the matrices
\begin{equation}
\begin{aligned}
J = J_N & := \begin{bmatrix} 0 &0 &0 \\ 0 & 0 & \Omega_N \\ 0 & -\Omega_N & 0 \end{bmatrix}, \\
\Omega_N & := \text{diag}\{\omega_0 I_r,2\omega_0 I_r,\ldots, (N-1)\omega_0 I_r\},\\
 B= B_N & := \sqrt{\frac{1}{\tau}} \begin{pmatrix} \dfrac{k_f^T}{\sqrt{2}} & k_f^T & \ldots & k_f^T & 0 & \ldots & 0
\end{pmatrix}^T,
\end{aligned}
\label{eq:harmonics}
\end{equation}
where $r= \text{rank } k_s$ and $k_f\in\mathbb{R}^{r \times p}$
satisfies $k_s = k_f^T k_f$. That the series (\ref{eq:kdistr})
converges in the sense of distributions means that for all smooth
$u(t)$ of support in $[0,\tau]$ we have that
\begin{equation*}
  k_su(t)  = \lim_{N\rightarrow \infty} \int_{-\infty}^{\infty}  \left(\kappa_N^{ac}(t-s) + \kappa_N^c(t-s) \right) u(s)ds.
\end{equation*}
A closer study of the two terms under the integral reveals that
\begin{align*}
   \lim_{N\rightarrow \infty} \int_{-\infty}^{\infty} \kappa_N^{ac}(t-s)u(s)ds  & = \frac{1}{2} k_s u(t+), \\
   \lim_{N\rightarrow \infty} \int_{-\infty}^{\infty}\kappa_N^c(t-s)u(s)ds  & = \frac{1}{2} k_s u(t-),
\end{align*}
because of the anti-causal/causal decomposition and
$\kappa_N^{c}(t)=\kappa_N^{ac}(-t),\, t>0$. Thus since $u(t)$ is
smooth, we can also model $y(t)=k_su(t)$ using only the causal
part $\kappa_n^c(t)$ if it is scaled by a factor of two. This
leads to a linear lossless approximation of $y(t)=k_su(t)$ that we
denote by the linear operator
$K_N:\mathcal{C}^2(0,\tau)\rightarrow \mathcal{C}^2(0,\tau)$
defined by
\begin{equation}
\begin{aligned}
  y_N(t) =(K_Nu)(t) &=  \int_{-\infty}^{\infty} 2\kappa_N^{c}(t-s)u(s) ds \\
  & = \int_0^t 2\kappa_N^{c}(t-s)u(s)ds.
\end{aligned}
\label{eq:kapprox}
\end{equation}
Here $\mathcal{C}^2(0,\tau)$ denotes the space of twice
continuously differentiable functions on the interval $[0,\tau]$.
The linear operator $K_N$ is realized by the triple
$(J_N,\sqrt{2}B_N,\sqrt{2}B_N^T)$. We can bound the approximation
error as seen in the following theorem.
\begin{theorem}
\label{prop:errbound} Assume that $u\in\mathcal C^2(0,\tau)$ and
$u(0)=0$. Let $y(t)=ku(t)=k_su(t)+k_au(t)$ with $k_s$ symmetric
positive semidefinite and $k_a$ anti symmetric. Define a lossless
approximation with realization
$(J_N,\sqrt{2}B_N,\sqrt{2}B_N^T,k_a)$, $y_N(t)=K_Nu(t) + k_a
u(t)$. Then the approximation error is bounded as
\begin{equation*}
  \|y(t)-y_N(t)\|_2 \leq \frac{2\bar{\sigma}(k_s)\tau}{\pi^2 (N-1)}\left( \|\dot u(t)\|_2+\|\dot u(0)\|_2+\|\ddot u \|_{L_1[0,t]}\right),
\end{equation*}
for $t$ in $[0,\tau]$.
\end{theorem}
\begin{proof}
We have that $y(t)-y_N(t) = \sum_{l=N}^{\infty}(2k_s/\tau)
\int_0^{t} \cos l\omega_0(t-s)u(s)ds$,  $t\in[0,\tau]$. The order
of summation and integration has changed because this is how the
value of the series is defined in distribution sense. We proceed
by using repeated integration by parts on each term in the series.
It holds that $\int_0^t \cos l\omega_0(t-s)u(s)ds = [\int_0^t \sin
l\omega_0(t-s)\dot u(s)ds]/(l\omega_0) = [\dot u(t)-\dot u(0) \cos
l\omega_0 t  -\int_0^t \cos l\omega_0(t-s)\ddot
u(s)ds]/(l^2\omega_0^2)$. Hence, we have the bound
\begin{multline*} \|y(t)-y_N(t)\|_2 \leq \frac{2\bar \sigma
(k_s)}{\tau} \sum_{l=N}^{\infty} \frac{1}{l^2 \omega_0^2}( \|\dot
u(t)\|_2
\\ +\|\dot u(0)\|_2+ \int_0^t \|\ddot u(s)\|_1 ds).
\end{multline*}
Since $\sum_{l=N}^{\infty}1/l^2 \leq 1/(N-1)$, we can establish
the bound in the theorem.
\end{proof}

The theorem shows that by choosing the truncation order $N$
sufficiently large, the memoryless model (\ref{eq:diss}) can be
approximated as well as we like with a lossless linear system, if
inputs are smooth. Hence we cannot then distinguish between the
systems $y=ku$ and $y_N=K_Nu+k_au$ using finite-time input-output
experiments. On physical grounds one may prefer the model
$K_N+k_a$ even though it is more complex, since it assumes the
form (\ref{eq:physlinear}) of a lossless system (and is time
reversible if $k$ is reciprocal, see Theorem~\ref{thm:timerev}).
Additional support for this idea is given in
Section~\ref{sec:fluct}. Note that the lossless approximation
$K_N$ is far from unique: The time interval $[0,\tau]$ is
arbitrary, and other Fourier expansions than (\ref{eq:kdistr}) are
possible to consider. The point is, however, that it is always
possible to approximate the dissipative behavior using a lossless
model.

It is often a reasonable assumption that inputs $u(t)$, for
example voltages, are smooth if we look at a sufficiently fine
time scale. This is because we usually cannot change inputs
arbitrarily fast due to physical limitations. Physically, we can
think of the approximation order $(2N-1)r$ as the number of
degrees of freedom in a physical system, usually of the order of
Avogadro's number, $N\approx 10^{23}$. It is then clear that the
interval length $\tau$ can be very large without making the
approximation error bound in Theorem~\ref{prop:errbound} large.
This explains how the dissipative system (\ref{eq:diss}) is
consistent with a physics based on energy conserving systems.

\begin{remark}
\label{rem:transmission} Note that it is well known that a
dissipative memoryless system can be modeled by an
\emph{infinite-dimensional} lossless system. We can model an
electrical resistor by a semi-infinite lossless transmission line
using the \emph{telegraphists's equation} (the wave equation), see
\cite{cheng}, for example. If the inductance and capacitance per
unit length of the line are $L$ and $C$, respectively, then the
characteristic impedance of the line, $\sqrt{L/C}$, is purely
resistive. One possible interpretation of $K_N$ is as a
finite-length lossless transmission line where only the $N$ lowest
modes of the telegraphists's equation are retained. Also in the
physics literature lossless (or Hamiltonian) approximations of
dissipative memoryless systems can be found. In
\cite{zwanzig73,ford+kac86,caldeira+83}, a so-called \emph{Ohmic
bath} is used, for example. Note that it is not shown in these
papers when, and how fast, the approximation converges to the
dissipative system. This is in contrast to the analysis presented
herein, and the error bound in Theorem~\ref{prop:errbound}.
\end{remark}

\subsection{Lossless approximation of dissipative systems with memory} \label{sec:memory} In
this section, we generalize the procedure from
Section~\ref{sec:staticio} to dissipative systems that have
memory. We consider asymptotically stable time-invariant linear
causal systems $G$ with impulse response
$g(t)\in\mathbb{R}^{p\times p}$. Their input-output relation is
given by
\begin{equation}
 y(t) = (Gu)(t) = \int_0^t g(t-s)u(s)ds.
\label{eq:Gio}
\end{equation}
Possible direct terms in $G$ can be approximated separately as
shown in Section~\ref{sec:staticio}. The system (\ref{eq:Gio}) is
dissipative with respect to the work rate $w(t)=y(t)^Tu(t)$ if and
only if $\int_0^\tau y(t)^Tu(t)dt \geq 0$, for all $\tau \geq 0$
and admissible $u(t)$. An equivalent condition, see
\cite{willems72B}, is that the transfer function satisfies
\begin{equation}
  \hat g(j\omega) + \hat g(-j\omega)^T \geq 0 \quad \text{for all} \quad \omega.
\label{eq:PR}
\end{equation}
Here $\hat g(j\omega)$ is the Fourier transform of $g(t)$.

We will next consider the problem of how well, and when, a system
(\ref{eq:Gio}) can be approximated using a linear lossless system
(\ref{eq:physlinear}) (call it $G_N$) with fixed initial state
$x_0$,
\begin{equation}
y_N(t) = B^T e^{Jt}x_0 + \int_0^t B^T e^{J(t-s)}B u(s)ds,
\label{eq:lossio}
\end{equation}
for a set of input signals. Let us formalize the problem.
\begin{problem}
For any fixed time horizon $[0,\tau]$ and arbitrarily small
$\epsilon>0$, when is it possible to find a lossless system with
fixed initial state $x_0$ and output $y_N$ such that
\begin{equation}
  \|y(t)-y_N(t)\|_2 \leq \epsilon \|u\|_{L_2[0,t]},
\label{eq:lossapprox}
\end{equation}
for all input signals $u \in L_2[0,t]$ and $0\leq t \leq \tau$?
\label{prob:1}
\end{problem}

Note that we require $x_0$ to be fixed in Problem~\ref{prob:1}, so
that it is independent of the applied input $u(t)$. This means the
approximation should work even if the applied input is not known
beforehand. Let us next state a necessary condition for linear
lossless approximations.
\begin{proposition}
\label{prop:initial} Assume there is a linear lossless system
$G_N$ that solves Problem~\ref{prob:1}. Then it holds that
\begin{itemize}
\item[(i)] If $x_0\neq 0$, then $x_0$ is an unobservable state;
\item[(ii)] If $x_0\neq 0$, then $x_0$ is an uncontrollable state; and
\item[(iii)] If the realization of $G_N$ is minimal, then
    $x_0=0$.
\end{itemize}
\end{proposition}
\begin{proof}
(i): The inequality (\ref{eq:lossapprox}) holds for $u=0$ when
$y=0$. Then (\ref{eq:lossapprox}) reduces to $\|y_N(t)\|_2 \leq
0$, for $t\in[0,\tau]$, which implies $y_N(t)=B^Te^{Jt}x_0=0$.
Thus a nonzero $x_0$ must be unobservable. (ii): For the lossless
realizations it holds that $\mathcal{N}(\mathcal{O}) =
\mathcal{R}(\mathcal{O}^T)^\perp= \mathcal{R}(\mathcal{C})^\perp$,
where $\mathcal{O}$ and $\mathcal{C}$ are the observability and
controllability matrices for the realization $(J,B,B^T)$. Thus if
$x_0$ is unobservable, it is also uncontrollable. (iii): Both (i)
and (ii) imply (iii).
\end{proof}

Proposition~\ref{prop:initial} significantly restricts the classes
of systems $G$ we can approximate using linear lossless
approximations. Intuitively, to approximate active systems there
must be energy stored in the initial state of $G_N$. But
Proposition~\ref{prop:initial} says that such initial energy is
not available for the inputs and outputs of $G_N$. The next
theorem shows that we can approximate $G$ using $G_N$ if, and only
if, $G$ is dissipative.
\begin{theorem}
\label{thm:diss2} Suppose $G$ is a linear time-invariant causal
system (\ref{eq:Gio}), where $\|g(t)\|_2$ is uniformly bounded,
$g(t)\in L_1 \cap L_2(0,\infty)$, and $\dot g(t)\in
L_1(0,\infty)$. Then Problem~\ref{prob:1} is solvable using a
linear lossless $G_N$ if, and only if, $G$ is dissipative.
\end{theorem}
\begin{proof}
See Appendix~\ref{sec:proofdiss2}.
\end{proof}

The proof of Theorem~\ref{thm:diss2} shows that the number of
states needed in $G_N$ is proportional to $\tau/\epsilon^2$, and
again the required state space is large. The result shows that for
finite-time input-output experiments with finite-energy inputs it
is not possible to distinguish between the dissipative system and
its lossless approximations. Theorem~\ref{thm:diss2} illustrates
that a very large class of dissipative systems (macroscopic
systems) can be approximated by the lossless linear systems we
introduced in (\ref{eq:physlinear}). The lossless systems are
dense in the dissipative systems, in the introduced topology.
Again this shows how dissipative systems are consistent with a
physics based on energy-conserving systems.

In \cite[Theorem~8]{willems72B}, necessary and sufficient
conditions for time reversible systems are given. We can now use
this result together with Theorem~\ref{thm:diss2} to prove a
result reminiscent to the \emph{Onsager  reciprocal relations}
which say physical systems tend to be reciprocal, see for example
\cite{wannier}. Before stating the result, we properly define what
is meant by reciprocal and  time reversible systems. These
definitions are slight reformulations of those found in
\cite{willems72B}.

A \emph{signature matrix} $\Sigma_e$ is a diagonal matrix with
entries  either $+1$ and $-1$.
\begin{definition}
\label{def:reciprocal} A linear time-invariant system $G$ with
impulse response $g(t)$ is \emph{reciprocal} with respect to the
signature matrix $\Sigma_e$ if $\Sigma_e g(t) = g(t)^T\Sigma_e$.
\end{definition}
\begin{definition}
\label{def:reversible} Consider a finite-dimensional linear
time-invariant system $G$ and assume that $x(0)=0$. Let $u_1,u_2$
be admissible inputs to $G$, and $y_1,y_2$ be the corresponding
outputs. Then $G$ is \emph{time reversible} with respect to the
signature matrix $\Sigma_e$ if $y_2(t)=\Sigma_e y_1(-t)$ whenever
$u_2(t)=-\Sigma_eu_1(-t)$.
\end{definition}

\begin{theorem}
\label{thm:timerev} Suppose $G$ satisfies the assumptions in
Theorem~\ref{thm:diss2}. Then $G$ is dissipative and reciprocal
with respect to $\Sigma_e$ if, and only if, there exists a
time-reversible (with respect to $\Sigma_e$) arbitrarily good
linear lossless approximation $G_N$ of $G$.
\end{theorem}
\begin{proof}
See Appendix~\ref{sec:prooftimerev}.
\end{proof}
Hence, one can understand that macroscopic physical systems close
to equilibrium usually are reciprocal because their underlying
dynamics are lossless \emph{and} time reversible.

\begin{remark}
There is a long-standing debate in physics about how macroscopic
time-irreversible dynamics can result from microscopic
time-reversible dynamics. The debate goes back to
\emph{Loschmidt's paradox} and the \emph{Poincar\'e recurrence
theorem}. The Poincar\'e recurrence theorem says that bounded
trajectories of volume-preserving systems (such as lossless
systems) will return arbitrarily close to their initial conditions
if we wait long enough (the Poincar\'e recurrence time). This
seems counter-intuitive for real physical systems. One common
argument is that the Poincar\'e recurrence time for macroscopic
physical systems is so long that we will never experience a
recurrence. But this argument is not universally accepted and
other explanations exist. The debate still goes on, see for
example \cite{Lebowitz99}. In this paper we construct lossless and
time-reversible systems with arbitrarily large Poincar\'e
recurrence times, that are consistent with observations of all
linear dissipative (time-irreversible) systems, as long as those
observations take place before the recurrence time. For a
control-oriented related discussion about the arrow of time, see
\cite{georgiou+08}.
\end{remark}

\subsection{Nonlinear lossless approximations}
\label{sec:nonlinapprox} In Section~\ref{sec:staticio}, it was
shown that a dissipative memoryless system can be approximated
using a lossless linear system. Later in Section~\ref{sec:memory}
it was also shown that the approximation procedure can be applied
to any dissipative (linear) system. Because of
Proposition~\ref{prop:initial} and Theorem~\ref{thm:diss2}, it is
clear that it is not possible to approximate a linear active
system using a \emph{linear} lossless system with fixed initial
state. Next we will show that it is possible to solve
Problem~\ref{prob:1} for active systems if we use \emph{nonlinear}
lossless approximations.

Consider the simplest possible active system,
\begin{equation}
  \label{eq:active_res}
  y(t) = k u(t),
\end{equation}
where $k\in\mathbb{R}^{p \times p}$ is negative definite. This can
be a model of a negative resistor, for example. More general
active systems are considered below. The reason a linear lossless
approximation of (\ref{eq:active_res}) cannot exist is that the
active device has an internal infinite energy supply, but we
cannot store any energy in the initial state of a linear lossless
system and simultaneously track a set of outputs, see
Proposition~\ref{prop:initial}. However, if we allow for lossless
nonlinear approximations, (\ref{eq:active_res}) can be arbitrarily
well approximated. This is shown next by means of an example.

Consider the nonlinear system
\begin{equation}
  \begin{aligned}
  \dot x_E(t) & = \frac{1}{\sqrt{2E_0}} u(t)^Tku(t), \quad x_E(0) = \sqrt{2E_0},\, E_0>0, \\
  y_E(t) & = \frac{x_E(t)}{\sqrt{2E_0}}ku(t),
  \end{aligned}
  \label{eq:nonlin_approx}
\end{equation}
with a scalar energy-supply state $x_E(t)$, and total energy
$E(x_E) = \frac{1}{2}x_E^2$. The system (\ref{eq:nonlin_approx})
has initial total energy $\frac{1}{2} x_E(0)^2 =: E_0$, and is a
lossless system with respect to the work rate $w(t)=y_E(t)u(t)$,
since
\begin{equation*}
    \frac{d}{dt} E(x_E(t)) = x_E(t)\dot x_E(t) = y_E(t)^Tu(t).
\end{equation*}
The input-output relation of (\ref{eq:nonlin_approx}) is given by
\begin{equation}
\begin{aligned}
  x_E(t) & = \sqrt{2E_0} + \frac{1}{\sqrt{2E_0}} \int_0^t u(s)^Tku(s) ds, \\
    y_E(t) &  = k u(t) + \frac{1}{2E_0}ku(t) \int_0^t u(s)^Tku(s)ds.
\end{aligned}
\label{eq:nonlin_approx_io}
\end{equation}
We have the following approximation result.
\begin{theorem}
\label{thm:nonlin_bound} For uniformly bounded inputs,
$\|u(t)\|_2\leq \bar u$, $t\in[0,\tau]$, the error between the
active system (\ref{eq:active_res}) and the nonlinear lossless
approximation (\ref{eq:nonlin_approx}) can be bounded as
\begin{equation*}
  \|y_E(t)-y(t)\|_2 \leq \epsilon \|u\|_{L_2[0,t]},
\end{equation*}
for $t\in[0,\tau]$, where $\epsilon = \bar \sigma (k)^2 \bar u^2
\sqrt{\tau}/(2E_0)$.
\end{theorem}
\begin{proof}
A simple bound on $y_E(t)-ku(t)$ from (\ref{eq:nonlin_approx_io})
gives  $\|y_E(t)-y(t)\|_2 \leq \frac{\bar \sigma
(k)^2\|u(t)\|_2}{2E_0}  \int_0^t \|u(s)\|_2^2 ds$. Then using
$\|u(t)\|_2\leq \bar u$, $t\in[0,\tau]$, gives the result.
\end{proof}

The error bound in Theorem~\ref{thm:nonlin_bound} can be made
arbitrarily small for finite time intervals if the initial total
energy $E_0$ is large enough. This example shows that active
systems can also be approximated by lossless systems, if the
lossless systems are allowed to be nonlinear and are charged with
initial energy.

The above approximation method can in fact be applied to much more
general systems. Consider the ordinary differential equation
\begin{equation}
\begin{aligned}
    \dot x(t) & = f(x(t),u(t)), \quad x(0) = x_{0}, \\
  y(t) & =  g(x(t),u(t)),
\end{aligned}
\label{eq:ODE}
\end{equation}
where $x(t)\in\mathbb{R}^n$, and $u(t),y(t)\in\mathbb{R}^p$. In
general, this is not a lossless system with respect to the supply
rate $w(t)=y(t)^Tu(t)$. A nonlinear lossless approximation of
(\ref{eq:ODE}) is given by
\begin{equation}
\begin{aligned}
    \dot{\hat x}(t) & = \frac{x_{E}(t)}{\sqrt{2E_0}} f(\hat x(t),u(t)), \quad \hat x(0) = x_{0}, \\
    \dot x_{E}(t) & = \frac{1}{\sqrt{2E_0}} g(\hat x(t),u(t))^Tu(t) - \frac{1}{\sqrt{2E_0}} \hat x(t)^T f(\hat x(t),u(t)), \\
    y_E(t) & =  \frac{x_{E}(t)}{\sqrt{2E_0}}g(\hat x(t),u(t)), \quad x_{E}(0) = \sqrt{2E_0}, \\
\end{aligned}
\label{eq:ODE_approx}
\end{equation}
where again $x_{E}(t)$ is a scalar energy-supply state, and $\hat
x(t)\in\mathbb{R}^n$ can be interpreted as an approximation of
$x(t)$ in (\ref{eq:ODE}). That (\ref{eq:ODE_approx}) is lossless
can be verified using the storage function
\begin{equation*}
E = \frac{1}{2} \hat x(t)^T\hat x(t) + \frac{1}{2} x_{E}(t)^2,
\end{equation*}
since
\begin{align*}
\dot E & = (x_{E}/\sqrt{2E_0}) (\hat x^T f(\hat x,u) + g(\hat x,u)^Tu -
\hat x^T f(\hat x,u)) \\
& = (x_{E}/\sqrt{2E_0}) g(\hat x,u)^Tu = y_E^T u = w.
\end{align*}
Since $x_{E}(t)/\sqrt{2E_0} \approx 1$ for small $t$, it is
intuitively clear that $\hat x(t)$ in (\ref{eq:ODE_approx}) will
be close to $x(t)$ in (\ref{eq:ODE}), at least for small $t$ and
large initial energy $E_0$. We have the following theorem.
\begin{theorem}
\label{thm:nonlin_bound2} Assume that $\partial f/\partial x$ is
continuous with respect to $x$ and $t$, and that (\ref{eq:ODE})
has a unique solution $x(t)$ for $0\leq t\leq \tau$. Then there
exist positive constants $C_1$ and $E_1$ such that for all
$E_0\geq E_1$ (\ref{eq:ODE_approx}) has a unique solution $\hat
x(t)$ which satisfies $\|x(t)-\hat x(t)\|_2 \leq C_1/\sqrt{2E_0}$
for all $0\leq t \leq \tau$.
\end{theorem}
\begin{proof}
Introduce the new coordinate $\Delta x_E=x_E-\sqrt{2E_0}$ and
define $\epsilon_0:=1/\sqrt{2E_0}$. The system
(\ref{eq:ODE_approx}) then takes the form
\begin{equation*}
\begin{aligned}
    \dot{\hat x} & = (1+\epsilon_0 \Delta x_E) f(\hat x,u), \quad \hat x(0) = x_{0}, \\
    \Delta \dot x_{E} & = \epsilon_0 g(\hat x,u)^Tu - \epsilon_0 \hat x^T f(\hat x,u), \quad \Delta x_E(0)=0.
\end{aligned}
\end{equation*}
Perturbation analysis \cite[Section~10.1]{Khalil} in the parameter
$\epsilon_0$ as $\epsilon_0 \rightarrow 0$ yields that there are
positive constants $\epsilon_1$ and $C_1$ such that $\|x-\hat
x\|_2\leq C_1|\epsilon_0|$ for all $|\epsilon_0|\leq \epsilon_1$.
The result then follows with $E_1=1/(2\epsilon_1^2)$.
\end{proof}
Just as in Section~\ref{sec:memory}, the introduced lossless
approximations are not unique. The one introduced here,
(\ref{eq:ODE_approx}), is very simple since only one extra state
$x_E$ is added. Its accuracy ($C_1$, $E_0$) of course depends on
the particular system ($f$, $g$) and the time horizon $\tau$. An
interesting topic for future work is to develop a theory for
``optimal'' lossless approximations using a fixed amount of energy
and a fixed number of states.

\subsection{Summary}
In Section~\ref{sec:micro}, we have seen that a large range of
systems, both dissipative and active, can be approximated by
lossless systems. Lossless systems account for the total energy,
and we claim these models are more physical. It was shown that
linear lossless systems are dense in the set of linear dissipative
systems. It was also shown that time reversibility of the lossless
approximation is equivalent to a reciprocal dissipative system. To
approximate active systems nonlinearity is needed. The introduced
nonlinear lossless approximation has to be initialized at a
precise state with a large total energy ($E_0$). The nonlinear
approximation achieves better accuracy (smaller $\epsilon$) by
increasing initial energy (increasing $E_0$). This is in sharp
contrast to the linear lossless approximations of dissipative
systems that are initialized with zero energy ($E_0=0$). These
achieve better accuracy (smaller $\epsilon$) by increasing the
number of states (increasing $N$). The next section deals with
uncertainties in the initial state of the lossless approximations.

\section{The Fluctuation-Dissipation Theorem}
\label{sec:fluct} As discussed in the introduction, the
fluctuation-dissipation theorem plays a major role in
close-to-equilibrium statistical mechanics. The theorem has been
stated in many different settings and for different models. See
for example \cite{Kubo,marconi+08}, where it is stated for
Hamiltonian systems and Langevin equations. In
\cite{twiss55,anderson82}, it is stated for electrical circuits. A
fairly general form of the fluctuation-dissipation theorem is
given in \cite[p.~500]{wannier}. We re-state this version of the
theorem here.

Suppose that $y_i$ and $u_i$, $i=1,\ldots,p$, are conjugate
external variables (inputs and outputs) for a dissipative system
in thermal equilibrium of temperature $T$ [Kelvin] (as defined in
Section~\ref{sec:fluctderiv}). We can interpret $y_i$ as a
generalized velocity and $u_i$ as the corresponding generalized
force, such that $y_iu_i$ is a work rate [Watt]. Although the
system is generally nonlinear, we only consider small variations
of the state around a fixpoint of the dynamics, which allows us to
assume the system to be linear. Assume first that the system has
no direct term (no memoryless element). If we make a perturbation
in the forces $u$, the velocities $y$ respond according to
\begin{equation*}
y(t) = \int_0^{t} g(t-s) u(s)ds,
\end{equation*}
where $g(t)\in\mathbb{R}^{p\times p}$ is the impulse response
matrix by definition. The following fluctuation-dissipation
theorem now says that the velocities $y$ actually also fluctuates
around the equilibrium.
\begin{proposition}
\label{prop:fluctdiss} The total response of a linear dissipative
system $G$ with no memoryless element and in thermal equilibrium
of temperature $T$ is given by
\begin{equation}
y(t) = n(t) + \int_0^{t} g(t-s) u(s)ds,
\label{eq:fluctdiss_phys}
\end{equation}
for perturbations $u$. The fluctuations $n(t)\in\mathbb{R}^p$ is a
stationary Gaussian stochastic process, where
\begin{equation}
\begin{aligned}
\mathbf{E} n(t)& =0, \\ R_n(t,s)& :=\mathbf{E} n(t)n(s)^T \\ &=
\left\{\begin{aligned}
&k_BT  g(t-s),\, t-s\geq 0\\
&k_BT  g(s-t)^T,\, t-s<0,
\end{aligned}
\right.
\end{aligned}
\label{eq:flucts}
\end{equation}
where $k_B$ is Boltzmann's constant.
\end{proposition}
\begin{proof}
See Section~\ref{sec:fluctderiv}.
\end{proof}

The covariance function of the noise $n$ is determined by the
impulse response $g$, and vice versa. The result has found
wide-spread use in for example fluid mechanics: By empirical
estimation of the covariance function we can estimate how the
system responds to external forces. In circuit theory, the result
is often used in the other direction: The forced response
determines the color of the inherent thermal noise. One way of
understanding the fluctuation-dissipation theorem is by using
linear lossless approximations of dissipative models, as seen in
the next subsection.

We may also express (\ref{eq:fluctdiss_phys}) in state space form
in the following way. A dissipative system with no direct term can
always be written as \cite[Theorem~3]{willems72B}:
\begin{equation}
\begin{aligned}
  \dot x(t) & = (J - K) x(t) + Bu(t),  \\
  y(t) & = B^T x(t),
\end{aligned}
\label{eq:dissipsystem}
\end{equation}
where $K=K^T$ is positive semidefinite and $J$ anti symmetric. To
account for (\ref{eq:fluctdiss_phys})--(\ref{eq:flucts}), it
suffices to introduce a white noise term $v(t)$ in
(\ref{eq:dissipsystem}) in the following way,
\begin{equation}
\begin{aligned}
  \dot x(t) & = (J - K) x(t) + Bu(t) + \sqrt{2k_BT} L v(t),  \\
  y(t) & = B^T x(t),
\end{aligned}
\label{eq:Langevin}
\end{equation}
where the matrix $L$ is chosen such that $LL^T=K$. Equation
(\ref{eq:Langevin}) is the called the Langevin equation of the
dissipative system.

Dissipative systems with memoryless elements are of great
practical significance. Proposition~\ref{prop:fluctdiss} needs to
be slightly modified for such systems.
\begin{proposition}
\label{prop:johnsonnoise} The total response of a linear
dissipative memoryless system in thermal equilibrium of
temperature $T$ and for perturbations $u$ is given by
\begin{equation}
y(t) = n(t) + ku(t) = n(t)+k_s u(t) + k_a u(t),
\label{eq:johnsonnoise}
\end{equation}
where $k_s\geq 0$ is symmetric positive semidefinite, and $k_a$
anti symmetric. The fluctuations $n(t)\in\mathbb{R}^p$ is a white
Gaussian stochastic process, where
\begin{equation*}
\begin{aligned}
\mathbf{E} n(t)& =0, \\
R_n(t,s)& :=\mathbf{E} n(t)n(s)^T = 2k_B T k_s   \delta(t-s).
\end{aligned}
\end{equation*}
\end{proposition}

Proposition~\ref{prop:johnsonnoise} follows from
Proposition~\ref{prop:fluctdiss} if one extracts the dissipative
term $k_su(t)$ from the memoryless model $ku(t)$ and puts
$g(t)=k_s\delta(t)$. However, the integral in
(\ref{eq:fluctdiss_phys}) runs up to $s=t$ and cuts the impulse
$\delta(t)$ in half. The re-normalized impulse response of the
dissipative term is therefore given by $g(t)=2k_s\delta(t)$ (see
also Section~\ref{sec:staticio}). The result then follows using
this $g(t)$ by application of Proposition~\ref{prop:fluctdiss}.
One explanation for why the anti symmetric term $k_a$ can be
removed from $g(t)$ is that it can be realized exactly using the
direct term $D$ in linear lossless approximation
(\ref{eq:physlinear}). An application of
Proposition~\ref{prop:johnsonnoise} gives the Johnson-Nyquist
noise of a resistor.
\begin{example}
As first shown theoretically in \cite{nyquist28} and
experimentally in \cite{johnson28}, a resistor $R$ of temperature
$T$ generates white noise. The total voltage over the resistor,
$v(t)$, satisfies $v(t)=Ri(t)+n(t)$, $\mathbf{E}n(t)n(s) =
2k_BTR\delta(t-s)$, where $i(t)$ is the current.
\end{example}

\subsection{Derivation using linear lossless approximations}
\label{sec:fluctderiv} Let us first consider systems without
memoryless elements. The general solution to the linear lossless
system (\ref{eq:physlinear}) is then
\begin{equation}
  y(t) = B^Te^{Jt}x_0 + \int_{0}^t B^Te^{J(t-s)}Bu(s)ds,
\label{eq:gensol}
\end{equation}
where $x_0$ is the initial state. It is the second term, the
convolution, that approximates the dissipative $(Gu)(t)$ in the
previous section. In Proposition~\ref{prop:initial}, we showed
that the first transient term is not desired in the approximation.
Theorems~\ref{prop:errbound} and \ref{thm:diss2} suggest that we
will need a system of extremely high order to approximate a linear
dissipative system on a reasonably long time horizon. When dealing
with systems of such high dimensions, it is reasonable to assume
that the exact initial state $x_0$ is not known, and it can be
hard to enforce $x_0=0$. Therefore, let us take a statistical
approach to study its influence. We have that
\begin{equation*}
  \mathbf{E}y(t) = B^Te^{Jt}\mathbf{E}x_0 + \int_{0}^t B^Te^{J(t-s)}Bu(s)ds,\quad t\geq 0,
\end{equation*}
if the input $u(t)$ is deterministic and $\mathbf{E}$ is the
expectation operator. The autocovariance function $R_y$ for $y(t)$
is then
\begin{equation}
\begin{aligned}
  R_{y}(t,s)&:= \mathbf{E}[y(t)-\mathbf E y(t)][y(s)-\mathbf E y(s)]^T \\ &= B^Te^{Jt}X_0 e^{-Js}B,
\end{aligned}
\label{eq:resnoise}
\end{equation}
where $X_0$ is the covariance of the initial state,
\begin{equation}
X_0 := \mathbf{E}\Delta x_0 \Delta x_0^T,
\label{eq:covariance}
\end{equation}
where $\Delta x_0:= x_0-\mathbf Ex_0$ is the stochastic uncertain
component of the initial state, which evolves as $\Delta x(t) =
e^{Jt}\Delta x_0$. The positive semidefinite matrix $X_0$ can be
interpreted as a measure of how well the initial state is known.
For a lossless system with total energy $E(x)=\frac{1}{2}x^Tx$ we
define the \emph{internal energy} as
\begin{equation}
U(x):= \frac{1}{2} \Delta x^T \Delta x, \quad \Delta x:= x - \mathbf{E}x.
\label{eq:internalenergy2}
\end{equation}
The expected total energy of the system equals $\mathbf{E}E(x)=
\frac{1}{2}(\mathbf{E}x)^T\mathbf{E}x + \mathbf{E}U(x)$. Hence the
internal energy captures the stochastic part of the total energy,
see also \cite{san07a,san07d}. In statistical mechanics, see
\cite{wannier,kittel+80,ma85}, the temperature of a system is
defined using the internal energy.
\begin{definition}[Temperature]
\label{def:temp} A system with internal energy $U(x)$ [Joule] has
temperature $T$ [Kelvin] if, and only if, its state $x$ belongs to
Gibbs's distribution with probability density function
\begin{equation}
p(x) = \frac{1}{Z}\exp[-U(x)/k_BT],
\label{eq:Gibbs}
\end{equation}
where $k_B$ is Boltzmann's constant and $Z$ is the normalizing
constant called the partition function. A system with temperature
is said to be at thermal equilibrium.
\end{definition}

When the internal energy function is quadratic and the system is
at thermal equilibrium, it is well known that the uncertain energy
is equipartitioned between the states, see
\cite[Sec.~4-5]{wannier}.
\begin{proposition}
\label{prop:temp}
 Suppose a lossless system with internal energy function
 $U(x)=\frac{1}{2}\Delta x^T \Delta x$ has temperature $T$
 at time $t=0$. Then the initial state $x_0$
 belongs to a Gaussian distribution with covariance matrix $X_0= k_BTI_n$, and
 $\mathbf{E}U(x_0)=\frac{n}{2}k_BT$.
\end{proposition}

Hence, the temperature $T$ is proportional to how much uncertain
equipartitioned energy there is per degree of freedom in the
lossless system. There are many arguments in the physics and
information theory literature for adopting the above definition of
temperature. For example, Gibbs's distribution maximizes the
Shannon continuous entropy (principle of maximum entropy
\cite{Jaynes57A,Jaynes57B}). In this paper, we will simply accept
this common definition of temperature, although it is interesting
to investigate more general definitions of temperature of
dynamical systems.
\begin{remark}
Note that lossless systems may have a temperature at any time
instant, not only at $t=0$. For instance, a lossless linear system
(\ref{eq:gensol}) of temperature $T$ at $t=0$ that is driven by a
deterministic input remains at the same temperature and has
constant internal energy at all times, since $\Delta x(t)$ is
independent of $u(t)$. To change the internal energy using
deterministic inputs, nonlinear systems are needed as explained in
\cite{Brockett+78,Brockett99}. For the related issue of entropy
for dynamical systems, see \cite{Brockett+78,san07a}.
\end{remark}

If a lossless linear system (\ref{eq:gensol}) has temperature $T$
at $t=0$ as defined in Definition~\ref{def:temp} and
Proposition~\ref{prop:temp}, then the autocovariance function
(\ref{eq:resnoise}) takes the form
\begin{equation*}
  R_{y}(t,s) = k_BT \cdot B^Te^{J(t-s)}B = k_BT \cdot [B^Te^{J(s-t)}B]^T,
\end{equation*}
since $J^T=-J$. It is seen that linear lossless systems satisfy
the fluctuation-dissipation theorem
(Proposition~\ref{prop:fluctdiss}) if we identify the stochastic
transient in (\ref{eq:gensol}) with the fluctuation, i.e.
$n(t)=B^T e^{Jt}x_0$ (assuming $\mathbf{E}x_0=0$), and the impulse
response as $g(t)=B^Te^{Jt}B$. In particular, $n(t)$ is a Gaussian
process of mean zero because $x_0$ is Gaussian and has mean zero.

Theorem~\ref{thm:diss2} showed that dissipative systems with
memory can be arbitrarily well approximated by lossless systems.
Hence we cannot distinguish between the two using only
input-output experiments. One reason for preferring the lossless
model is that its transient also explains the thermal noise that
is predicted by the fluctuation-dissipation theorem. To explain
the fluctuation-dissipation theorem for systems without memory
(Proposition~\ref{prop:johnsonnoise}), one can repeat the above
arguments by making a lossless approximation of $k_s$ (see
Theorem~\ref{prop:errbound}). The anti symmetric part $k_a$ does
not need to be approximated but can be included directly in the
lossless system by using the anti symmetric direct term $D$ in
(\ref{eq:lossapprox}).

Proposition~\ref{prop:johnsonnoise} captures the notion of a
\emph{heat bath}, modelling it (as described in
Theorem~\ref{prop:errbound}) with a lossless system so large that
for moderate inputs and within  the chosen time horizon, the
interaction with its environment is not significantly affected.

That the Langevin equation (\ref{eq:Langevin}) is a valid
state-space model for (\ref{eq:fluctdiss_phys}) is shown by a
direct calculation. If we assume that (\ref{eq:dissipsystem}) is a
low-order approximation for a high-order linear lossless system
(\ref{eq:gensol}), in the sense of Theorem \ref{thm:diss2}, it is
enough to require that both systems are at thermal equilibrium
with the same temperature $T$ in order to be described by the same
stochastic equation (\ref{eq:fluctdiss_phys}), at least in the
time interval in which the approximation is valid.

\subsection{Nonlinear lossless approximations and thermal noise}
Lossless approximations are not unique. We showed in
Section~\ref{sec:nonlinapprox} that low-order nonlinear lossless
approximations can be constructed. As seen next, these do
\emph{not} satisfy the fluctuation-dissipation theorem. This is
not surprising since they can also model active systems.  If they
are used to implement linear dissipative systems, the linearized
form is not in the form (\ref{eq:physlinear}). By studying the
thermal noise of a system, it could in principle be possible to
determine what type of lossless approximation that is used.

Consider the nonlinear lossless approximation
(\ref{eq:nonlin_approx}) of $y(t)=ku(t)$, where $k$ is scalar and
can be either positive or negative. The approximation only works
well when the initial total energy $E_0$ is large. To study the
effect of thermal noise, we add a random Gaussian perturbation
$\Delta x_0$ to the initial state so that the system has
temperature $T$ at $t=0$ according to Definition~\ref{def:temp}
and Proposition~\ref{prop:temp}. This gives the system
\begin{equation}
  \begin{aligned}
  \dot x_E(t) & = \frac{k}{\sqrt{2E_0}} u(t)^2, \quad x_E(0) = \sqrt{2E_0}+\Delta x_0,\,\mathbf{E} \Delta x_0  = 0, \\
  y_E(t) & = \frac{k}{\sqrt{2E_0}}x_E(t)u(t), \quad  \mathbf{E}\Delta x_0^2 = k_BT.
  \end{aligned}
\label{eq:nonlin_approx_temp}
\end{equation}
The solution to the lossless approximation
(\ref{eq:nonlin_approx_temp}) is given by
\begin{equation}
    y_E(t)  = k u(t) + n_s(t) + n_d(t),
\label{eq:nonlinenoisedef0}
\end{equation}
where
\begin{equation}
n_d(t)=\frac{k^2}{2E_0}u(t)\int_0^tu(s)^2ds, \quad n_s(t)=\frac{k\Delta x_0}{\sqrt{2E_0}}u(t).
\label{eq:nonlinenoisedef}
\end{equation}
We call $n_d(t)$ the deterministic implementation noise and
$n_s(t)$ the stochastic thermal noise. The ratio between the
deterministic and stochastic noise is
\begin{equation*}
\frac{n_d(t)}{n_s(t)}  =  \frac{k}{\sqrt{2E_0}\Delta x_0} \int_0^t u(s)^2 ds = \frac{ku(0)^2}{\sqrt{2E_0}\Delta x_0} t + O(t^2),
\end{equation*}
as $t\rightarrow 0$, if $u(t)$ is continuous. Hence, for
sufficiently small times $t$ and if $\Delta x_0\neq 0$, the
stochastic noise $n_s(t)$ is the dominating noise in the lossless
approximation (\ref{eq:nonlin_approx_temp}). Since $\Delta x_0$
belongs to a Gaussian distribution, there is zero probability that
$\Delta x_0 = 0$. Hence, the solution $y_E(t)$  can be written
\begin{equation}
\begin{aligned}
    y_E(t)  &= k u(t) + n_s(t) + O(t),\\
    \mathbf{E}n_s(t) &=0,\quad \mathbf{E}n_s(t)^2 = \frac{k^2k_B T}{2E_0}u(t)^2.
    \end{aligned}
\label{eq:nonlin_approx_temp_sol}
\end{equation}
Just as in Proposition~\ref{prop:johnsonnoise}, the noise
variance is proportional to the temperature $T$. Notice, however,
that the noise is significantly smaller in
(\ref{eq:nonlin_approx_temp_sol}) than in
Proposition~\ref{prop:johnsonnoise}. There the noise is white and
unbounded for each $t$. The expression
(\ref{eq:nonlin_approx_temp_sol}) is further used in
Section~\ref{sec:hardlimit}.

\subsection{Summary}
In Section~\ref{sec:fluct}, we have seen that uncertainty in the
initial state of a linear lossless approximation gives a simple
explanation for the fluctuation-dissipation theorem. We have also
seen seen that uncertainty in the initial state of a nonlinear
lossless approximation gives rise to noise which does not satisfy
the fluctuation-dissipation theorem. In all cases, the variance of
the noise is proportional to the temperature of the system. Only
when the initial state is perfectly know, that is when the system
has temperature zero, perfect approximation using lossless systems
can be achieved.

\section{Limits on Measurements and Back Action}
\label{sec:hardlimit} In this section, we study measurement
strategies and devices using the developed theory. In quantum
mechanics, the problem of measurements and their interpretation
have been much studied and debated. Also in classical physics
there have been studies on limits on measurement accuracy. Two
examples are \cite{barnes+34,mccombie53}, where thermal noise in
measurement devices is analyzed and bounds on possible measurement
accuracy derived. Nevertheless, the effect of the measurement
device on the measured system, the ``back action", is usually
neglected in classical physics. That such effects exist also in
classical physics is well known, however, and is called the
``observer effect". Also in control engineering these effects are
usually neglected: The sensor is normally modeled to interact with
the controlled plant only through the feedback controller.

Using the theory developed in this paper, we will quantify and
give limits on observer effects in a fairly general setting. These
limitations should be of practical importance for control systems
on the small physical scale, such as for MEMS and in systems
biology.

\subsection{Measurement problem formulation}
Assume that the problem is to estimate the scalar potential
$y(t_m)$ (an output) of a linear dissipative dynamical system
$\mathcal{S}$ at some time $t_m>0$. Furthermore, assume that the
conjugate variable of $y$ is $u$ (the ``flow" variable). Then the
product $y(t)u(t)$ is a work rate. As has been shown in
Section~\ref{sec:memory}, all single-input--single-output linear
dissipative systems can be arbitrarily well approximated by a
dynamical system in the form,
\begin{equation}
\mathcal{S}: \quad \left\{
\begin{aligned}
    \dot x(t) & = Jx(t) + Bu(t), & x(0)& =x_0, \\
    y(t) & = B^T x(t), &  y(0)&=y_0 = B^Tx_0,
  \end{aligned}
\right.
\label{eq:unpert}
\end{equation}
for a fixed initial state $x_0$. Note that this system evolves
deterministically since $x_0$ is fixed. Let us also define the
parameter $C$ by $B^TB =: 1/C$. Then $1/C$ is the first Markov
parameter of the transfer function of $\mathcal{S}$. If
$\mathcal{S}$ is an electrical capacitor and the measured quantity
a voltage, $C$ coincides with the capacitance.

\begin{figure}[tb]
\centering
\psfrag{tm}[][]{$[0,t_m]$}
\psfrag{S}[][]{$\mathcal{S}$}
\psfrag{R}[][]{$\mathcal{M}$}
\psfrag{u}[][]{$u(t)$}
\psfrag{um}[][]{$u_m(t)$}
\psfrag{y}[][]{$y(t)$}
\includegraphics[width=0.6\hsize]{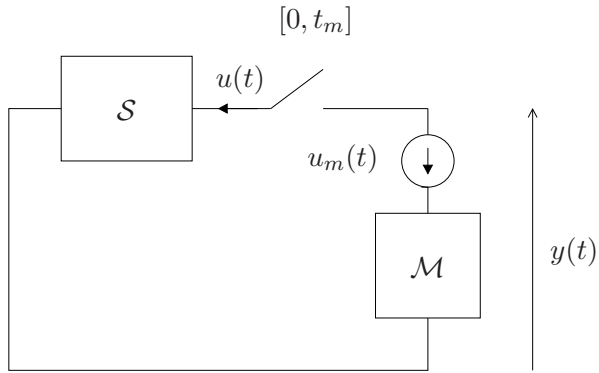}
\caption{Circuit diagram of an idealized measurement device $\mathcal{M}$ and the measured system $\mathcal{S}$. The
measurement is performed in the time interval $[0,t_m]$. The problem is to estimate the potential $y(t_m)$ as well as possible, given the
flow measurement $u_m=-u$.}
\label{fig:measideal}
\end{figure}
To estimate the potential $y(t_m)$, an idealized measurement
device called $\mathcal{M}$ is connected to $\mathcal{S}$ in the
time interval $[0,t_m]$, see Fig.~\ref{fig:measideal}. The
validity of Kirchoff's laws is assumed in the interconnection.
That is, the flow out of $\mathcal{S}$ goes into $\mathcal{M}$,
and the potential difference $y(t)$ over the devices is the same
(a lossless interconnection). The device $\mathcal{M}$ has an
ideal flow meter that gives the scalar value $u_m(t)=-u(t)$.
Therefore the problem is to estimate the potential of
$\mathcal{S}$ given knowledge of the flow $u(t)$. For this
problem, two related effects are studied next, the \emph{back
action} $b(t_m)$, and the \emph{estimation error} $e(t_m)$. By
back action we mean how the interconnection with $\mathcal{M}$
effects the state of $\mathcal{S}$. It quantifies how much the
state of $\mathcal{S}$ deviates from its natural trajectory after
the measurement. Estimation error is the difference between the
actual potential and the estimated potential. Next we consider two
measurement strategies and their lossless approximations in order
to study the impact of physical implementation.
\begin{figure}[tb]
\centering
\psfrag{tm}[][]{$[0,t_m]$}
\psfrag{S}[][]{$\mathcal{S}$}
\psfrag{R}[][]{$k_m$}
\psfrag{-R}[][]{$-k_m$}
\psfrag{u}[][]{$u(t)$}
\psfrag{um}[][]{$u_m(t)$}
\psfrag{y}[][]{$y(t)$}
\includegraphics[width=1.0\hsize]{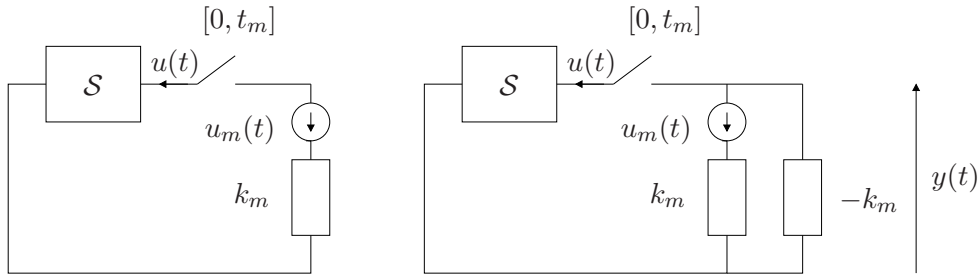}
\caption{Circuit diagrams of the memoryless dissipative measurement device $\mathcal{M}_1$ (left) and
the memoryless active measurement device $\mathcal{M}_2$ (right).}
\label{fig:meascirc}
\end{figure}

\begin{remark}
\label{rem:measass} The reason the initial state $x_0$ in
$\mathcal{S}$ is fixed is that we want to compare how different
measurement strategies succeed when used on exactly the same
system. We also assume that $y_0=B^Tx_0$ is  completely unknown to
the measurement device before the measurement starts.
\end{remark}

\subsection{Memoryless dissipative measurement device}
\label{sec:dissmeas} Consider the measurement device
$\mathcal{M}_1$ to the left in Fig.~\ref{fig:meascirc}. This
measurement device connected to $\mathcal{S}$ is modeled by a
memoryless system with (a known) admittance $k_m>0$,
\begin{equation*}
  \mathcal{M}_1:\left\{
\begin{aligned}
 u_m(t)  & = -u(t) = k_m y(t) \\
 y_m(t) & = \frac{u_m(t)}{k_m} = y(t). \\
\end{aligned}
\right.
\end{equation*}
The signal $y_m(t)$ is the measurement signal produced by
$\mathcal{M}_1$.
 The dynamics of the interconnected measured system becomes
\begin{equation*}
  \mathcal{SM}_1: \left\{
\begin{aligned}
   \dot x_1(t) & = (J-k_mBB^T)x_1(t), \quad x_1(0)=x_0, \\
   y_1(t)& = B^T x_1(t),
    \end{aligned} \right.
\end{equation*}
where $x_1(t)$ is the state of $\mathcal{S}$ when it is
interconnected to $\mathcal{M}_1$. If the measurement circuit is
closed in the time interval $[0,t_m]$, then the state of the
system $\mathcal{S}$ gets perturbed from its natural trajectory by
a quantity
\begin{align*}
  b(t_m)  & := x_1(t_m)-x(t_m) = e^{(J-k_mBB^T)t_m}x_0 -  e^{Jt_m}x_0  \\ & = - k_m y_0 B t_m + O(t_m^2),
\end{align*}
where $x(t)$ satisfies (\ref{eq:unpert}) with $u(t)=0$, and
$b(t_m)$ is the back action. By making the measurement time $t_m$
small, the back action can be made arbitrarily small.

In this situation, a good estimation policy for the potential
$y_1(t_m)$ is to choose $\hat y(t_m)=y_m(t_m)$, since the
estimation error $e(t_m)$ is identically zero in this case,
\begin{equation*}
  e(t_m):= \hat y(t_m)- y_1(t_m)= 0.
\end{equation*}
The signal $\hat y(t_m)$ should here, and in the following, be
interpreted as the best possible estimate of the potential of
$\mathcal S$ for someone who has access to the measurement signal
$y_m(t)$, $0\leq t \leq t_m$. Note that the estimation error $e$
is defined with respect to the perturbed system $\mathcal{SM}_1$.
Given that we already have defined back action it is easy to give
a relation to the unperturbed system $\mathcal{S}$ by
\begin{equation}
y(t_m) = \hat y(t_m) - e(t_m) - B^Tb(t_m),
\label{eq:correction}
\end{equation}
which is valid for non-zero estimation errors also.
\begin{remark}
Whether one is interested in the perturbed potential $y_1(t_m)$ or
the unperturbed potential $y(t_m)$ of $\mathcal{S}$ depends on the
reason for the measurement. For a control engineer who wants to
act on the measured system,  $y_1(t_m)$ is likely to be of most
interest. A physicist, on the other hand, who is curious about the
uncontrolled system may be more interested in $y(t_m)$. Either
way, knowing the back action $b$, one can always get $y(t_m)$ from
$y_1(t_m)$ using (\ref{eq:correction}).
\end{remark}

\subsubsection{Lossless realization $\hat{\mathcal{M}}_1$}
\label{sec:lossessM1} Next we make a linear lossless realization
of the admittance $k_m>0$ in $\mathcal{M}_1$, using
Proposition~\ref{prop:johnsonnoise}, so that it satisfies the
fluctuation-dissipation theorem. Linear physical implementations
of $\mathcal{M}_1$ inevitably exhibit this type of Johnson-Nyquist
noise. We obtain
\begin{equation*}
  \hat{\mathcal M}_{1}:\left\{
\begin{aligned}
 u_m(t)  & = -u(t) = k_m y(t) + \sqrt{2k_mk_BT_m}n(t), \\
 y_m(t) & = \frac{u_m(t)}{k_m} = y(t) + \sqrt{\frac{2k_BT_m}{k_m}}n(t), \\
\end{aligned}
\right.
\end{equation*}
where $T_m$ is the temperature of the measurement device, and
$n(t)$ is unit-intensity white noise. As shown before, the noise
can be interpreted as due to our ignorance of the exact initial
state of the measurement device. The interconnected measured
system $\mathcal{S} \hat{\mathcal M}_{1}$ satisfies a
Langevin-type equation,
\begin{equation*}
  \mathcal{S} \hat{\mathcal M}_{1}: \left\{
\begin{aligned}
  \dot x_{1}(t) & = (J-k_mBB^T)x_{1}(t) - \sqrt{2k_m k_BT_m}Bn(t), \\ x_{1}(0)& =x_0, \\
y_1(t) & = B^T x_1(t).
\end{aligned}
 \right.
\end{equation*}
The solution for $\mathcal{S}\hat {\mathcal M}_{1}$ is
\begin{multline*}
x_{1}(t)  = e^{(J-k_m B B^T )t} x_0 \\ - \int_0^t e^{(J-k_m B
B^T)(t-s)} B \sqrt{2k_mk_BT_m} n(s) ds.
\end{multline*}
The back action can be calculated as
\begin{align*}
b(t_m) & = x_1(t_m) - x(t_m)= b_d(t_m) + b_s(t_m) , \\
b_d(t_m) & := \mathbf{E} x_{1}(t_m)-x(t_m) = e^{(J-k_mBB^T)t_m}x_0 -  e^{Jt_m}x_0  \\ & = - k_m  y_0 B t_m + O(t_m^2), \\
b_s(t_m)  & := x_{1}(t_m) - \mathbf{E} x_{1}(t_m) \\ & = - \int_0^{t_m} e^{(J-k_m BB^T)(t_m-s)} B \sqrt{2k_mk_BT_m} n(s) ds,
\end{align*}
where we have split the back action into deterministic and
stochastic parts. The deterministic back action coincides with
the back action for $\mathcal{M}_1$. The stochastic back action
comes from the uncertainty in the lossless realization of the
measurement device. The measurement device $\hat{\mathcal{M}}_1$
injects a stochastic perturbation into the measured system
$\mathcal{S}$.

The covariance $P$ of the back action $b$ at time $t_m$ is
\begin{multline} P(t_m)  := \mathbf{E}
[b(t_m)-\mathbf{E}b(t_m)][b(t_m)-\mathbf{E}b(t_m)]^T \\ =
\mathbf{E} b_s(t_m) b_s(t_m)^T    = 2k_mk_BT_m \int_0^{t_m}
e^{(J-k_m BB^T)(t_m-s)}B\\ \times B^T (e^{(J-k_m BB^T)(t_m-s)})^T
ds   = 2BB^T k_mk_B T_m t_m + O(t_m^2).\label{eq:covmeas}
\end{multline}
It holds that $P(t_m)\rightarrow k_B T_m  I_n$ and
$\mathbf{E}x_1(t) \rightarrow 0$ as $t_m \rightarrow \infty$, see
\cite[Propositions~1 and 2]{san07d}, and the measured system
attains temperature $T_m$ after an infinitely long measurement. It
is therefore reasonable to keep $t_m$ small if one wants to have a
small back action.

Next we analyze and bound the estimation error. The measurement
equation is given by
\begin{align*}
    y_m(t) & = \frac{u_m(t)}{k_m} = y_{1}(t) + \sqrt{\frac{2k_B T_m}{k_m}}n(t).
\end{align*}
Note that $\hat y(t_m) = y_m(t_m)$ is now a poor estimator of
$y_1(t_m)$, since the variance of the estimation error $e(t)=\hat
y(t)-y_1(t)$ is infinite due to the white noise $n(t)$. Using
filtering theory, we can construct an optimal estimator that
achieves a fundamental lower bound on the possible accuracy
(minimum variance) given $y_{m}(t)$ in the interval $0\leq t\leq
t_m$. The solution is the Kalman filter,
\begin{equation}
\begin{aligned}
   \dot{\hat x}_{1}(t) & = (J-k_mBB^T)\hat x_{1}(t) + K(t)[y_m(t)-B^T \hat x_{1}(t)],  \\
    \hat y(t) & = B^T \hat{x}_{1}(t),
\end{aligned}
\label{eq:kalman}
\end{equation}
where $K(t)$ is the Kalman gain (e.g. \cite{astrom}). The minimum
possible variance of the estimation error, $M^*(t_m) = \min
\mathbf{E}(\hat y(t_m)- y_1(t_m))^2$ ($^*$ denotes optimal) can be
computed from the differential Riccati equation
\begin{align}
\dot X(t) & = J_{k_m} X(t)+X(t)J_{k_m}^T + 2k_mk_BT_mBB^T \nonumber \\
 & - \frac{k_m}{2k_BT_m}(X(t) - 2k_BT_m I_n)B \nonumber \\ & \qquad \qquad \times B^T (X(t)-2k_BT_mI_n)^T, \label{eq:ricc} \\
M^*(t_m) & = B^T X(t_m)B, \quad J_{k_m}:=J-k_mBB^T.  \nonumber
\end{align}
A series expansion $X(t)=\frac{1}{t}X_{-1}+X_0 + tX_1+\ldots$ of
the solution to (\ref{eq:ricc}) yields that the coefficient
$X_{-1}$ should satisfy $X_{-1}=\frac{k_m}{2k_BT_m}X_{-1}BB^T
X_{-1}$. Note that $X_{-1}$ is independent on $J_{k_m}$. From the
$X_1$ equation, we yield that
\begin{equation*}
M^*(t_m) = \frac{2k_BT_m}{k_m t_m} + O(1),
\end{equation*}
since $M^*(t)=\frac{1}{t}B^TX_{-1}B+ B^TX_0B + tB^TX_1B+\ldots$
Here the boundary condition $M^*(0)=+\infty$ has been used, since
it is assumed that $y_0$ is  completely unknown, see
Remark~\ref{rem:measass}. It is easy to verify that $M^*(t_m)
\rightarrow 0$ as $t_m \rightarrow \infty$, and given an
infinitely long measurement a perfect estimate is obtained. This
comes at the expense of a large back action.

To implement the Kalman filter (\ref{eq:kalman}) requires a
complete model ($J,B,k_m,T_m$) which is not always reasonable to
assume. Nevertheless, the Kalman filter is optimal and the
variance of the estimation error, $M(t):=\mathbf{E}e(t)^2$, of any
other estimator, in particular those that do not require complete
model knowledge, must satisfy
\begin{equation}
M(t_m)\geq M^*(t_m) = \frac{2k_BT_m}{k_m t_m} + O(1).
\label{eq:covest}
\end{equation}

\subsubsection{Back action and estimation error trade-off}
Define the root mean square back action and the root mean square
estimation error of the potential $y$ by
\begin{equation*}
  |\Delta y(t_m)| := \sqrt{B^T P(t_m) B}, \quad |\Delta \hat y(t_m)|:= \sqrt{M(t_m)}.
\end{equation*}
This is the typical magnitude of the change of the potential $y$
and the estimation error after a measurement. Using
(\ref{eq:covmeas}) and (\ref{eq:covest}), the appealing relation
\begin{equation}
  |\Delta y(t_m)| |\Delta \hat y(t_m)| \geq 2 k_BT_m / C + O(t_m),
\label{eq:hardlimit}
\end{equation}
where $1/C = B^TB$, is obtained. Hence, there is a direct
trade-off between the accuracy of estimation and the
perturbation in the potential, independently on (small) $t_m$
and admittance $k_m$. It is seen that the more ``capacitance"
$(C)$ $\mathcal{S}$ has, the less important the trade-off is.
One can interpret $C$ as a measure of the physical size or
inertia of the system. The trade-off is more important for
``small" system in ``hot" environments. Using an optimal filter,
the trade-off is satisfied with equality.

\subsection{Memoryless active measurement device}
\label{sec:actmeas} A problem with the device $\mathcal{M}_1$ is
that it causes back action $b$ even in the most ideal situation.
If active elements are allowed in the measurement device, this
perturbation can apparently be easily eliminated, but of course
with the inherent costs of an active device. Consider the
measurement device $\mathcal{M}_2$ to the right in
Fig.~\ref{fig:meascirc}. It is modeled by
\begin{equation*}
  \mathcal{M}_2:\left\{
\begin{aligned}
 u_m(t) & = k_my(t), \\
 u(t)  & =  u_m(t)-k_my(t) = 0, \\
 y_m(t) & = \frac{u_m(t)}{k_m} = y(t), \\
\end{aligned}
\right.
\end{equation*}
where an active element $-k_m$ exactly compensates for the back
action in $\mathcal{M}_1$. It is clear that there is no back
action and no estimation error using this device,
\begin{equation*}
  \label{eq:6}
  b(t_m)= 0, \quad e(t_m) = 0,
\end{equation*}
for all $t_m$. Next, a lossless approximation of $\mathcal{M}_2$
is performed.

\begin{table*}[!tb]
\caption{Summary of back action and estimation error after a
measurement in the time interval $[0,t_m]$. $b_d(t_m)$ -
deterministic back action, $P(t_m)$ - covariance of back action,
$|\Delta y|^2$ - variance of potential, and $M^*(t_m)$ - lower
bound on estimation error.} \centering
\makebox[5cm][c]{\begin{tabular}{l|cccc}
\hline
Device & $b_d(t_m)$ & $ P(t_m)=\mathbf{E}b_s(t_m)b_s(t_m)^T $ & $|\Delta y(t_m)|^2=B^TP(t_m)B$ & $M^*(t_m)=\min |\Delta \hat y|^2$ \\ \hline
$\mathcal{M}_1$ & $-k_my_0 Bt_m + O(t_m^2)$ & 0 &  0 & 0 \\
$\hat{\mathcal{M}}_{1}$ & $- k_m y_0B t_m + O(t_m^2)$ & $2 k_mk_BT_m BB^T t_m + O(t_m^2)$ & $\frac{2k_mk_B T_m }{C^2}t_m + O(t_m^2)$ & $\frac{2k_BT_m}{k_m }t_m^{-1} + O(1)$ \\ \hline
$\mathcal{M}_2$ & 0 & 0 & 0 & 0 \\
$\hat{\mathcal{M}}_{2}$ & $\frac{y_0^3 k_m}{4E_m}B t_m^2 + O(t_m^3)$ & $2k_mk_BT_m BB^Tt_m + O(t_m^2)$ & $\frac{2k_m k_BT_m}{C^2} t_m + O(t_m^2)$ &
$\frac{2k_BT_m}{k_m} t_m^{-1} +O(1)$
\end{tabular}}
\label{table:back1}
\end{table*}

\subsubsection{Lossless realization $\hat{\mathcal{M}}_2$}
Let the dissipative element $k_m$ in $\mathcal{M}_2$ be
implemented with a linear lossless system, see
Proposition~\ref{prop:johnsonnoise}, and the active element $-k_m$
be implemented using the nonlinear lossless system in
(\ref{eq:nonlin_approx_temp}). This approximation of
$\mathcal{M}_2$ captures the reasonable assumption that the
measurement device must be charged with energy to behave like an
active device, and that its linear dissipative element satisfies
the fluctuation-dissipation theorem.

Assume that the temperature of the measurement device
$\hat{\mathcal{M}}_2$ is $T_m$ and the deterministic part of the
total energy of the active element is $E_m$. Then the
interconnected system becomes
\begin{equation*}
 \mathcal{S} \hat{\mathcal{M}}_{2}:\left\{
\begin{aligned}
\dot x_{2}(t) & = (J-k_mBB^T)x_{2}(t) \\ &+ \frac{k_m}{\sqrt{2E_m}} x_r(t) BB^Tx_{2}(t) \\ & - B \sqrt{2k_mk_BT_m}n(t), \quad x_2(0)=x_0, \\
\dot x_r(t) & = \frac{k_m}{\sqrt{2E_m}} (B^T x_{2}(t))^2, \\
x_r(0) &= \sqrt{2E_m} + \Delta x_{r0}, \\ \mathbf{E}\Delta x_{r0} &=0, \quad \mathbf{E}\Delta x_{r0}^2 =k_BT_m, \\
y_m(t) & = \frac{u_m(t)}{k_m} = B^Tx_2(t) + \sqrt{\frac{2k_BT_m}{k_m}}n(t),
\end{aligned}
\right.
\end{equation*}
where $x_2$ is the state of $\mathcal{S}$, and $x_r$ is the state
of the active element. Using the closed-form solution
(\ref{eq:nonlinenoisedef0})--(\ref{eq:nonlinenoisedef}) to
eliminate $x_r$, we can also write the equations as
\begin{equation}
 \mathcal{S} \hat{\mathcal{M}}_{2}:\left\{
\begin{aligned}
\dot x_{2}(t) & = \left( J+\frac{k_m \Delta x_{r0}}{\sqrt{2E_m}}BB^T \right) x_{2}(t) \\ &  + Bw_d(t) - B \sqrt{2k_mk_BT_m}n(t), \,x_2(0)=x_0, \\
y_m(t) & = \frac{u_m(t)}{k_m} = B^Tx_2(t) + \sqrt{\frac{2k_BT_m}{k_m}}n(t),
\end{aligned}
\right.
\label{eq:ShatM2}
\end{equation}
with the deterministic perturbation $w_d(t)=\frac{k_m^2
y_0^3}{2E_m}t +O(t^2)$. The solution to (\ref{eq:ShatM2}) can be
expanded as
\begin{multline} x_{2}(t) = x_0 - \sqrt{2k_mk_BT_m} B N(t) \\ + \left( J +
\frac{k_m \Delta x_{r0}}{\sqrt{2E_m}}BB^T \right) x_0 t
\\ - \sqrt{2k_mk_BT_m} \left( J+\frac{k_m \Delta
x_{r0}}{\sqrt{2E_m}}BB^T\right)\\
\times B \int_0^t N(s)ds + B\frac{k_m^2 y_0^3}{4E_m}t^2 + o(t^2),
\end{multline}
where $N(t)=\int_0^t n(s)ds = O(\sqrt{t})$ is integrated white
noise (a Brownian motion). It can be seen that the white noise
disturbance $n$ is much more important than the deterministic
disturbance $w_d$.
 The back action becomes
\begin{align*}
b(t_m) & =  x_2(t_m) - x(t_m)=b_d(t_m)+b_s(t_m) \\
b_{d}(t_m) & := \mathbf{E} x_{2}(t_m)-x(t_m)  = \frac{k_m^2 y_0^3}{4E_m}Bt_m^2 + O(t_m^3), \\
 b_{s}(t_m)  & := x_{2}(t_m) - \mathbf{E} x_{2}(t_m) \\ & = -\sqrt{2k_mk_BT_m} B N(t_m) + \frac{k_m
\Delta x_{r0}}{\sqrt{2E_m}}By_0 t_m \\ & \qquad \qquad + O(t_m\sqrt{t_m}),
\end{align*}
where we used that the covariance between $\Delta x_{r0}$ and $N$
is zero. The covariance of the back action becomes
\begin{equation}
P(t_m) := \mathbf{E}b_{s}(t_m)b_{s}(t_m)^T =  2k_mk_BT_m BB^Tt_m + O(t_m^2).
\end{equation}
It is seen that the dominant term in the stochastic back action is
the same as for $\hat{\mathcal{M}}_1$, but the deterministic back
action $b_d$ is much smaller.

\begin{remark} Using a
nonlinear lossless approximation of $-k_m$ of order larger than
one, we can make the deterministic back action smaller for fixed
$E_m$, at the expense of model complexity.
\end{remark}

The measurement noise in $\mathcal{S} \hat{\mathcal{M}}_2$ is the
same as in $\mathcal{S} \hat{\mathcal{M}}_1$, and we can
essentially repeat the argument from Section~\ref{sec:lossessM1}.
The difference between $\mathcal{S} \hat{\mathcal{M}}_2$ and
$\mathcal{S} \hat{\mathcal{M}}_1$ lies in the dynamics. In
$\mathcal{S} \hat{\mathcal{M}}_2$, the system matrix is
$J+\frac{k_m \Delta x_{r0}}{\sqrt{2E_m}}BB^T$ and there is a
deterministic perturbation $w_d(t)$. To make an estimate $\hat
y(t_m)$, knowledge of $y_m(t)$ in the interval $[0,t_m]$ is
assumed. If we assume that the model $(J,B,k_m,T_m)$ is known plus
that the observer somehow knows $w_d(t)$ and $\Delta x_{r0}$, then
the optimal estimate again has the error covariance $M^*(t_m) =
\frac{2k_BT_m}{k_m t_m} + O(1)$. Any other estimator that has less
information available must be worse, so that
\begin{equation*}
M(t) \geq M^*(t_m) = \frac{2k_BT_m}{k_m t_m} + O(1).
\end{equation*}
Again, we have the trade-off (\ref{eq:hardlimit})
\begin{equation*}
  |\Delta y(t_m)| |\Delta \hat y(t_m)| \geq 2 k_BT_m / C + O(t_m),
\end{equation*}
which holds even though we have inserted an active element in
device. The only effect of the active element is to eliminate the
deterministic back action.

\subsection{Summary and Discussion}
The back action and estimation error of the measurement devices
are summarized in Table~\ref{table:back1}. For the ideal devices
$\mathcal{M}_1$ and $\mathcal{M}_2$ no real trade-offs exist.
However, if we realize them with lossless elements very reasonable
trade-offs appear. It is only in the limit of infinite available
energy and zero temperature that the trade-offs disappear. The
deterministic back action can be made small with large $E_m$,
charging the measurement device with much energy. However, the
effect of stochastic back action is inescapable for both
$\hat{\mathcal{M}}_1$ and $\hat{\mathcal{M}}_2$, and the trade-off
\begin{equation}
  |\Delta y||\Delta \hat y| \geq 2 k_BT_m/C \quad \text{for small } t_m,
\label{eq:hardlimit2}
\end{equation}
holds in both cases. The reason for having short measurements is
to minimize the effect of the back action. The lower bound on the
estimation error $M^*(t_m)$ tends to zero for large $t_m$, but at
the same time the measured system $\mathcal{S}$ tends to a
thermodynamic equilibrium with the measurement device.

It is possible to increase the estimation accuracy by making the
admittance $k_m$ of the measurement device large, but only at the
expense of making a large stochastic perturbation of the measured
system. Hence, we have quantified a limit for the observer effect
discussed in the introduction of this section. We conjecture that
inequalities like (\ref{eq:hardlimit2}) hold for very general
measurement devices as soon as the dissipative elements satisfy
the fluctuation-dissipation theorem.  Note, for example, that if a
lossless transmission cable of admittance $k_m$ and of temperature
$T_m$ is used to interconnect the system $\mathcal{S}$ to an
arbitrary measurement device $\mathcal{M}$, then the trade-off
(\ref{eq:hardlimit2}) holds. The deterministic back action, on the
other hand, is possible to make smaller by using more elaborate
nonlinear lossless implementations.

\section{Conclusions}
In this paper, we constructed lossless approximations of both
dissipative and active systems.  We obtained an if-and-only-if
characterization of linear dissipative systems (linear lossless
systems are dense in the linear dissipative systems) and gave
explicit approximation error bounds that depend on the time
horizon, the order, and the available energy of the
approximations. We showed that the fluctuation-dissipation
theorem, that quantifies macroscopic thermal noise, can be
explained by uncertainty in the initial state of a linear lossless
approximation of very high order. We also saw that using these
techniques, it was relatively easy to quantify limitations on the
back action of measurement devices. This gave rise to a trade-off
between process and measurement noise.

\section{Appendices}
\subsection{Proof of Theorem~\ref{thm:diss2}}
\label{sec:proofdiss2} We first show the 'only if' direction.
Assume the opposite: There is a lossless approximation $G_N$ that
satisfy (\ref{eq:lossapprox}) for arbitrarily small $\epsilon>0$
even though $G$ is not dissipative. From
Proposition~\ref{prop:initial} it is seen that we can without loss
of generality assume $G_N$ has a minimal realization and $x_0=0$.
If $G$ is not dissipative, we can find an input $u(t)$ over the
interval $[0,\tau]$ such that $\int_0^\tau y(t)^Tu(t)dt = -K_1 <
0$, i.e., we extract energy from $G$ even though its initial state
is zero. Call $\|u\|_{L_1[0,\tau]}=K_2$. We have $\int_0^\tau
(y_N(t)-y(t))^Tu(t)dt \leq  \epsilon K_2$, by the assumption that
a lossless approximation $G_N$ exists and using the Cauchy-Schwarz
inequality. But the lossless approximation satisfies $\int_0^\tau
y_N(t)^Tu(t)dt = \frac{1}{2}x(\tau)^Tx(\tau)$, since $x_0=0$.
Hence, $-\int_0^\tau y(t)^Tu(t)dt=K_1 \leq \epsilon K_2-
\frac{1}{2}x_{N}(\tau)^Tx_{N}(\tau)  \leq \epsilon K_2$. But since
$\epsilon$ can be made arbitrarily small, this leads to a
contradiction.

To prove the 'if' direction we explicitly construct a $G_N$ that
satisfies (\ref{eq:lossapprox}), when $G$ is dissipative. It turns
out that we can fix the model parameters $D=0$ in $G_N$.
Furthermore, we must choose $x_0=0$ since otherwise the zero
trajectory $y=0$ cannot be tracked (see above). We thus need to
construct a lossless system with impulse response $g_{N}(t)$ such
that $\|g-g_N\|_{L_2[0,\tau_0]}\leq \epsilon$, where we have
denoted the time interval given in the theorem statement by
$[0,\tau_0]$. Note that we can increase this time interval without
loss of generality, since if we prove $\|g-g_N\|_{L_2[0,\tau]}
\leq \epsilon$ then $\|g-g_N\|_{L_2[0,\tau_0]}\leq \epsilon$, if
$\tau \geq \tau_0$.

Let us define the constants
\begin{align*}
 C_1 & \geq \|g(t)\|_2, \quad t\geq 0; \quad &  C_2 &= \int_0^{\infty}\|\dot g(t)\|_1 dt; \\
 C_3& =\int_0^\infty \|g(t)\|_1 dt; \quad & C&=\frac{4C_1+2C_2}{\pi}+\frac{4C_3}{\tau_0},
\end{align*}
which are all finite by the assumptions of the theorem. It will
become clear later why the constants are defined this way.

Next let us fix the approximation time interval $[0,\tau]$ such
that
\begin{equation}
\delta(\tau):= \int_\tau^{\infty} \|g(t)\|_1 dt \leq \frac{\epsilon^2}{2C\sqrt{p}},
\label{eq:taufixed}
\end{equation}
where $\tau\geq \tau_0$. Such a $\tau$ always exists since
$\delta(\tau)$ is a continuously decreasing function that
converges to zero. The lossless approximation is achieved by
truncating a Fourier series keeping $N$ terms. Let us choose the
integer $N$ such that
\begin{equation}
 N \leq \frac{\tau C^2}{\epsilon^2} \leq N+1,
\label{eq:Nfixed}
\end{equation}
where $\tau$ is fixed in (\ref{eq:taufixed}). We proceed by
constructing an appropriate Fourier series.

\subsubsection{Fourier expansion} The extended function $\tilde
g(t)\in L_2(-\infty,\infty)$ of $g(t)$ is given by
\begin{equation*}
\tilde g(t) =
\left\{
\begin{aligned}
g(t), & \quad  t \geq 0, \\
g(-t)^T, & \quad  t < 0.
\end{aligned}
\right.
\end{equation*}
Let us make a Fourier expansion of $\tilde g(t)$ on the interval
$[-\tau,\tau]$,
\begin{equation*}
\tilde g_{\tau}(t) := \frac{1}{2}A_0 + \sum_{k=1}^{\infty} A_k \cos \frac{k\pi t}{\tau} + B_k \sin  \frac{k\pi t}{\tau},
\end{equation*}
with convergence in $L_2[-\tau,\tau]$. For the restriction to
$[0,\tau]$ it holds that $\|g-\tilde g_{\tau}\|_{L_2[0,\tau]}=0$.
The expressions for the (matrix) Fourier coefficients are
\begin{equation}
\begin{aligned}
A_k & = \frac{1}{\tau} \int_{0}^{\tau} (g(t)+g(t)^T) \cos \frac{k\pi t}{\tau} dt\\
B_k & = \frac{1}{\tau} \int_{0}^{\tau} (g(t)-g(t)^T) \sin \frac{k\pi t}{\tau} dt.
\end{aligned}
\label{eq:fouriercoeffs}
\end{equation}
Note that $A_k,B_k  \in\mathbb{R}^{p\times p}$, and $A_k$ are
symmetric ($A_k=A_k^T$) and $B_k$ are anti-symmetric
($B_k=-B_k^T$). Parseval's formula becomes
\begin{multline}
\|\tilde g_{\tau}\|_{L_2[0,\tau]}^2  = \int_0^{\tau} \text{Tr
}g(t)g(t)^T dt \\  = \frac{\tau}{4}\text{Tr }A_0^TA_0 +
\frac{\tau}{2} \sum_{k=1}^\infty \text{Tr }A_k^TA_k + \text{Tr
}B_k^TB_k. \label{eq:parseval}
\end{multline}

We also need to bound $\|A_k-jB_k\|_2^2 = \text{Tr }A_k^TA_k +
\text{Tr }B_k^TB_k$. It holds
\begin{align*}
A_k-jB_k & = \frac{1}{\tau} \int_{-\tau}^{\tau} \tilde g(t) e^{-j\pi k t/\tau}dt \\
& = \frac{(-1)^k}{jk\pi}(g(\tau)^T - g(\tau) ) + \frac{1}{jk\pi} (g(0)-g(0)^T) \\
& \quad \quad + \frac{1}{jk\pi} \int_0^\tau e^{-j\pi k t/\tau}\dot g(t) - e^{j\pi k t/\tau}\dot g(t)^T dt,
\end{align*}
using integration by parts.  Then
\begin{multline*}
\|A_k-jB_k \|_2   \\
\leq \frac{4C_1}{k\pi} +\frac{1}{k\pi} \left|\!\left|\int_0^\tau
e^{-j\pi k t/\tau}\dot g(t) - e^{j\pi k t/\tau}\dot g(t)^T dt
\right|\!\right|_2
 \\ \leq \frac{4C_1}{k\pi} + \frac{2}{k\pi} \int_0^\tau \|\dot{g}(t)\|_1 dt  \leq \frac{1}{k} \frac{4C_1+2C_2}{\pi}.
\end{multline*}
Furthermore,
\begin{equation*}
\|A_k-jB_k\|_2  = \left|\!\left|\frac{1}{\tau} \int_{-\tau}^{\tau} \tilde g(t) e^{-j\pi k t/\tau}dt \right|\!\right|_2 \leq \frac{2C_3}{\tau_0},
\end{equation*}
since $\tau \geq \tau_0$. If the former bound is multiplied by $k$
and the latter is multiplied by two and they are added together,
we obtain
\begin{equation} \|A_k-jB_k\|_2  \leq \frac{C}{2+k},
\quad k\geq 0, \label{eq:AkBkbound}
\end{equation}
where $C$ was defined above.

\subsubsection{Lossless approximation $G_N$} Let us now truncate
the series $\tilde g_\tau(t)$ and keep the terms with Fourier
coefficients $A_0,\ldots,A_{N-1}$ and $B_1,\ldots,B_{N-1}$. The
truncated impulse response can be realized exactly by a
finite-dimensional lossless system iff $A_0\geq 0$ and
$A_k-jB_k\geq 0$, $k=1,\ldots,N-1$, see
\cite[Theorem~5]{willems72B}. But these inequalities are not
necessarily true. We will thus perturb the coefficients to ensure
the system becomes lossless and yet ensure that the
$L_2$-approximation error is less than $\epsilon$.

We quantify a number $\xi \geq 0$ that ensures that $A_k-jB_k +
\xi I_p \geq 0$ for all $k$. Note that by the assumption of $G$
being dissipative, it holds that
\begin{equation*}
\hat{g}(j\omega )+\hat g(-j\omega )^T = \int_{-\infty}^{\infty} \tilde g(t) e^{-j\omega t}dt \geq 0.
\end{equation*}
Remember that $\int_{-\tau}^{\tau} \tilde g(t) e^{-j\pi k
t/\tau}dt = \tau A_k - j \tau B_k$, and therefore
\begin{equation*}
A_k -j B_k + \Delta_k \geq 0
\end{equation*}
where $\Delta_k := \frac{1}{\tau}\int_\tau^\infty g(t)e^{-j\pi k
t/\tau} + g(t)^T e^{j\pi k t/\tau}dt$. The size of $\Delta_k$ can
be bounded and we have
\begin{align*}
\|\Delta_k \|_2  = \sqrt{\text{Tr }\Delta_k^*\Delta_k}
\leq  \frac{2}{\tau}\int_\tau^{\infty} \|g(t)\|_1 dt \leq \frac{\epsilon^2}{\tau C\sqrt{p}},
\end{align*}
using (\ref{eq:taufixed}).  Thus we can choose
\begin{equation*}
\xi = \frac{\epsilon^2}{\tau C\sqrt{p}},
\end{equation*}
and $A_k-jB_k + \xi I_p \geq 0$ for all $k$, since $\rho(\Delta_k)
\leq \|\Delta_k\|_2$.

Next we verify that a system $G_N$ with impulse response
\begin{multline}
g_{N}(t) := \frac{1}{2}(A_0+\xi I_p) \\ + \sum_{k=1}^{N-1}
(A_k+\xi I_p) \cos \frac{k\pi t}{\tau} + B_k \sin  \frac{k\pi
t}{\tau}, \label{eq:lossapproxgen}
\end{multline}
where $\tau$, $N$, $\xi$ are fixed above satisfies the statement
of the theorem. By the construction of $\xi$, $G_N$ is lossless.
It remains to show that the approximation error
$\|g-g_N\|_{L_2[0,\tau]}$ is less than $\epsilon$. Using
Parseval's formula (\ref{eq:parseval}), it holds
\begin{multline*}
\|g-g_N\|_{L_2[0,\tau]}^2 = \|\tilde g_\tau -g_N\|_{L_2[0,\tau]}^2 \\
=\left|\!\left|\frac{1}{2}\xi I + \sum_{k=1}^{N-1} \xi I \cos
\frac{k\pi t}{\tau} +\sum_{k=N}^{\infty} A_k \cos \frac{k\pi t}{\tau} + B_k \sin  \frac{k\pi t}{\tau}\right|\!\right|_2^2\\
\leq \frac{\tau}{2} N\xi^2 p + \frac{\tau}{2} \sum_{k=N}^\infty
\|A_k-jB_k\|_2^2
 \leq \frac{\tau}{2} N\xi^2 p + \frac{\tau}{2}\sum_{k=N}^{\infty}
 \frac{C^2}{(2+k)^2} \\
 \leq \frac{\tau}{2} \frac{\tau C^2}{\epsilon^2} \frac{\epsilon^4}{\tau^2 C^2 p} p + \frac{\tau}{2}\frac{C^2}{N+1}
\leq \frac{\epsilon^2}{2} + \frac{\tau}{2}\frac{C^2
\epsilon^2}{\tau C^2} = \epsilon^2,
\end{multline*}
where the bounds (\ref{eq:Nfixed}) and (\ref{eq:AkBkbound}) are
used. The result has been proved.

\subsection{Proof of Theorem~\ref{thm:timerev}}
\label{sec:prooftimerev} We first show the 'if' direction. Then
there exists a lossless and time-reversible (with respect to
$\Sigma_e$, see Definition~\ref{def:reversible}) approximation
$G_N$ of $G$. Theorem~\ref{thm:diss2} shows that $G$ is
dissipative. Theorem~8 in \cite{willems72B} shows that $G_N$
necessarily is reciprocal with respect to $\Sigma_e$. Since $G_N$
is an arbitrarily good approximation it follows that $G$ also is
reciprocal, which concludes the 'if' direction of the proof.

Next we show the 'only if' direction. Then $G$ is dissipative and
reciprocal with respect to $\Sigma_e$. Theorem~\ref{thm:diss2}
shows that there exists an arbitrarily good lossless approximation
$G_N$, and we will use the approximation (\ref{eq:lossapproxgen}).
That $G$ is reciprocal with respect to $\Sigma_e$ means that
$\Sigma_e g(t) = g(t)^T\Sigma_e$, see
Definition~\ref{def:reciprocal}. Using this and the definition of
$A_k$ and $B_k$ in (\ref{eq:fouriercoeffs}), it is seen that
\begin{equation*}
\Sigma_e(A_k + \xi I_p )  = (A_k + \xi I_p )^T \Sigma_e, \quad
\Sigma_e B_k = B_k ^T \Sigma_e.
\end{equation*}
Thus the chosen $G_N$ is also reciprocal, $\Sigma_e
g_N(t)=g_N(t)^T\Sigma_e$, and Theorem~8 in \cite{willems72B} shows
$G_N$ is time reversible with respect to $\Sigma_e$. This
concludes the proof.

\section*{Acknowledgment} The authors would like to thank Dr.
B.~Recht for helpful suggestions and comments on an early version
of the paper, and Prof.~J.~C.~Willems for helpful discussions.


\end{document}